\documentclass{amsart}

\usepackage{latexsym,amssymb,amscd,textcomp,fancyhdr,lastpage}
\usepackage{amsbsy,amsmath,amsthm,amsfonts}
\usepackage[english]{babel}
\usepackage[all]{xy}
\usepackage[latin1]{inputenc}
\usepackage{array}
\usepackage{enumerate}

\newtheorem{Atheorem}{Theorem}

\newtheorem{theorem}{Theorem}[section]

\newtheorem{lemma}[theorem]{Lemma}
\newtheorem{corollary}[theorem]{Corollary}
\newtheorem{proposition}[theorem]{Proposition}

\newtheorem{definition}[theorem]{Definition}

\newtheorem{remark}[theorem]{Remark}
\theoremstyle{remark}
\newtheorem{rem}[theorem]{Remark}
\newtheorem{nb}[theorem]{Nota Bene}

\newtheorem*{NB*}{Nota Bene}

\newenvironment{sproof}[1]
{\begin{proof}[#1]} {\end{proof}}

\newcommand{\N}{\mathbb N}

\newcommand{\C}{\mathbb C}
\newcommand{\Z}{\mathbb Z}
\newcommand{\R}{\mathbb R}
\newcommand{\Q}{\mathbb Q}

\newcommand{\Free}{\mathbb{F}}
\newcommand{\F}{{\mathbb F}_2}

\newcommand{\G}{{\mathcal G}_2}

\newcommand{\BS}{\overline{BS}(m,\xi)}
\newcommand{\ovBS}{\overline{BS}(m,\xi)}
\newcommand{\bBS}[2]{\overline{BS}(#1,#2)}
\newcommand{\tBS}[2]{\widetilde{BS}(#1,#2)}
\newcommand{\eps}{\varepsilon}
\newcommand{\BBS}{\overline{BS}}

\newcommand{\br}[1]{\lbrack #1 \rbrack}
\newcommand{\Pres}[2]{\left\langle{#1}\ \big\vert\ {#2}\right\rangle}

\newcommand{\DS}{\textnormal{DSPACE}}
\newcommand{\NS}{\textnormal{NSPACE}}
\newcommand{\DT}{\textnormal{DTIME}}
\newcommand{\A}{\mathcal{A}}
\newcommand{\Aa}{\mathcal{A}^{\ast}}
\newcommand{\Mp}{\text{M}'}
\newcommand{\tv}{\tilde{v}}
\newcommand{\tw}{\tilde{w}}
\newcommand{\ab}{ \{a^{\pm1},b^{\pm1}\}^{\ast} }
\newcommand{\aeo}{ \{a^{\pm1},\pm e_0\}^{\ast} }

\newcommand{\mh}{\hat{m}}
\newcommand{\nh}{\hat{n}}

\newcommand{\roots}{\operatorname{Root}}

\title{Limits of Baumslag-Solitar groups and dimension estimates in the space of marked groups}
\author{Luc Guyot and Yves Stalder}


\begin{document}

\maketitle


\begin{abstract}

  We prove that the limits of Baumslag-Solitar groups studied in \cite{GS08} 
are non-linear hopfian C*-simple groups with infinitely many twisted conjugacy classes. 
We exhibit infinite presentations for these groups, classify them up to group isomorphism, describe their automorphisms and discuss the word and conjugacy problems. Finally, we prove that the set of these groups has non-zero Hausforff dimension in the space of marked groups on two generators.
\end{abstract}


\section*{Introduction} \label{SecIntro}

Baumslag-Solitar groups are ubiquitous in group theory and topology \cite{BS62,Mol69,Gil79,FM98,Why01,JS79} and offer a remarkable test bed for group-theoretic properties. Considering their limits, we obtain in the present paper a Cantor set of pairwise non-isomorphic 
two-generated groups with a number of combinatorial and geometrical non-closed properties. By a closed property, we mean a property that defines a closed subset of the space of marked groups. We also give the first estimates of non-vanishing Hausdorff dimension in the space of marked groups.

Let $m\in \Z \setminus \{0\}$. For every sequence $(\xi_n)$ of integers in $\Z$ such that $|\xi_n|$ tends to infinity and $\xi_n$ tends to some $m$-adic integer $\xi$, define $$\ovBS=\lim_{n \to \infty} BS(m,\xi_n)$$ where $BS(p,q)=\Pres{a,b}{ab^pa^{-1}=b^q}$ denotes the Baumslag-Solitar group with $p,\,q$ in $\Z \setminus \{0\}$; the limit is taken with respect to the topology of the space of marked groups \cite[Th. 6]{Sta06b}.
As a consequence of its definition, $\ovBS$ enjoys the following closed properties: it is torsion-free, centerless, non-Kazhdan\footnote{The property (T) of Kahzdan is open by \cite[Th. 6.7]{Sha00}.}, contains a non-abelian free group generated by $b$ and $bab^{-1}$ if $|m|>1$, and satisfies the relation\footnote{This is actually the shortest relation with respect to $a$ and $b$ by \cite[Pr. 2]{Sta06a}.} ``$\br{ ab^ma^{-1},\,b}=1$". We present here results of a different nature, relying on the existence of an HNN decomposition.

The limits $\ovBS$ are first studied for their own right in \cite{GS08}, where it is shown that $\ovBS$ acts transitively on a tree and maps homomorphically onto the special limit $\bBS{1}{0}=\Z \wr \Z$. These two results are extensively used in this article.  
\paragraph{Results}
 We first prove that $\BS$ is a non-degenerate HNN extension of a free abelian group of infinite countable rank. More precisely, consider the free abelian groups 
\begin{eqnarray*}
 E & = & \Z e_0 \oplus \Z e_1 \oplus \Z e_2 \oplus \cdots \\
 E_{m,\xi} & = & \Z me_0 \oplus \Z(e_1 - r_1(\xi)e_0) \oplus \Z(e_2 - r_2(\xi)e_0) \cdots \\
 E_1 & = & \Z e_1 \oplus \Z e_2 \oplus \cdots
\end{eqnarray*}
where $(r_i(\xi))$ is an integer sequence defined in Section \ref{SubSecRiSi}. 
Let $\widetilde{BS}(m,\xi)$ be the HNN extension of basis $E$
with conjugated subgroups $E_{m,\xi},\,E_1$ and stable letter $a$, where
conjugacy from $E_{m,\xi}$ to $E_1$ is defined by $a(me_0)a^{-1} =
e_1$ and $a(e_i-r_i(\xi)e_0)a^{-1} = e_{i+1}$.
\begin{Atheorem}[Cor. \ref{CorHNN2} and Th. \ref{ThPres}]  \label{ThmA}
Let $(a,b)$ be the canonical generating pair of $\ovBS$. Then the map defined by $f(a)=a$ and $f(b)=e_0$ induces an isomorphism from
$\ovBS$ to $\tBS{m}{\xi}$.
Moreover, the group $\ovBS$ admits the following infinite presentation:
$$
\Pres{a,b}{\br{b,b_i}=1,\,i\ge 1}
$$
with $b_1=ab^ma^{-1}$ and $b_{i}=ab_{i-1}b^{-r_{i-1}(\xi)}a^{-1}$
for every $i \ge 2$.

In particular, we have $\bBS{m}{0}=\Pres{a,b}{\br{a^ib^ma^{-i},\, b}=1,\, i \ge 1}$. In addition, the
latter presentation is minimal.
\end{Atheorem}

It follows from \cite[Th. 4.1]{GS08} that no group $\ovBS$ can be finitely presented.
Note also that $\ovBS$ enjoys any property shared by all non-degnerate HNN extensions, for instance it is primitive \cite{GG08} and has uniform exponential growth \cite{HB00}. Using the latter HNN decomposition, we show the following:
\begin{Atheorem}  \label{ThmB}
Assume that $|m|>1$. Then we have:
\begin{itemize}
\item $\ovBS$ is hopfian but not co-hopfian (Th. \ref{ThHopf}). If $m$ is prime and $\xi$ is algebraic over $\Q$, then $\ovBS$ is not residually finite (Pr. \ref{PropResFin}).
\item $\ovBS$ is C*-simple and inner-amenable (Pr. \ref{ThPowers} and Pr. \ref{PropInnAmen}).
\item $\ovBS$ has infinitely many twisted conjugacy classes (Pr. \ref{CorTwist}).
\item $\ovBS$ is not equationally noetherian and hence not linear (Pr. \ref{eqNbarBS}).
\item The automorphism group of $\ovBS$ is a split extension of $\ovBS$ by an infinite dihedral group (Pr. \ref{PropAut}).
\end{itemize}
\end{Atheorem}

From our study of homomorphisms carried out in Section \ref{SecHomo}, we deduce the following classification result.

\begin{Atheorem}[Th. \ref{ThClass} and Pr. \ref{PropRiSiCont}]
 Let $m,\,m' \in \Z\setminus \{0\}$ and let $ \xi \in \Z_m,\, \xi'\in \Z_{m'}$.
Then $\ovBS$ is abstractly isomorphic to $\bBS{m'}{\xi'}$ if and only if there is $\epsilon \in \{\pm 1\}$ and $d \in \N$ such that $m = \epsilon m'$, $d=\gcd(m,\xi)=\gcd(m',\xi')$ and the $m$-adic numbers $\xi/d,\,\epsilon \xi'/d$ project onto the same element of $\Z_{m/d}$ via the canonical map $\Z_{m} \longrightarrow \Z_{m/d}$.
\end{Atheorem}

Thus two given limits are isomorphic if and only if they are isomorphic as marked groups \cite[Th. 2.1]{GS08}, i.e. if and only if they represent the same point in the space of marked groups on two generators.

The first-named author has shown that the box-counting dimension of $\G$, the space of marked groups on two generators, is infinite \cite{Guy07}.
Estimating the Hausdorff dimension of $Z_m^\times$, the set of marked groups $\ovBS$ such that $\xi$ is invertible in $\Z_m$ (Th. \ref{ThDimZm}), we deduce the following:
\begin{Atheorem}[Cor. \ref{CorDimNotZero}]
 The Hausdorff dimension of $\G$ satisfies $\dim_H(\G) \geqslant \log(2)/6$. In particular, this dimension does not vanish.
\end{Atheorem}

Our last result pertains to the algorithmic complexity of the word and conjugacy problem (Prop. \ref{PropWPovBS}) and their degrees of undecidability. For every problem that can be suitably represented by a language and every $m$-adic number there is a well-defined degree of undecidability, called the \emph{Turing degree} (Sec. \ref{SecWP}).
\begin{Atheorem}[Cor. \ref{CorTuring}] Assume that $\xi$ is invertible in $\Z_m$. 
Then the following Turing degrees coincide:
\begin{itemize}
\item the Turing degree of the word problem for $\ovBS$;
\item the Turing degree of the conjugacy problem for $\ovBS$;
\item the Turing degree of $\xi$.
\end{itemize}
In particular, the word problem is solvable for $\ovBS$ if and only if $\xi$ is a computable number.
\end{Atheorem}

The \emph{resolution degree $r_{\Gamma}(n)$} of the word problem for a finitely generated group $\Gamma$ \cite[Def. 1]{Gri85} is a quantitative measure of the undecidability of the word problem based on Kolgomorov's ideas. It is, intuitively, the minimal amount of information necessary to decide if $w=1$ in $G$ for every word $w$ with $|w| \le n$, where $|w|$ is the length of $w$ with respect to a chosen generating set.
Let $\omega$ be an infinite sequence of symbols from some finite alphabet. We denote by $\omega^{(n)}$ be the word made of the first $n$ symbols of $\omega$. For a word $w$, we denote by $KR(w)$ the \emph{Kolmogorov complexity resolution} of $w$ which is, intuitively, the minimal amount of information necessary to obtain, for every natural number $i \le |w|$, the $i$-th symbol of the word $w$. Two functions $f$ and $g$ are said to be \emph{equivalent} if there is some $C \ge 0$ such that $f(n) \le g(Cn)+C$ and $g(n) \le f(Cn)+C$ for every $n$.
\begin{Atheorem}[Pr. \ref{PropRiSiCont} and Th. \ref{CongDistLim}]
 Let $\Gamma=\ovBS$ and let $\omega=(r_i(\xi))$. Then $r_{\Gamma}(n)$ is equivalent to $KR(\omega^{(n)})$.
\end{Atheorem}
 For comparison, the resolution degree of the Grigorchuk group $\Gamma=G_{\omega}$ is equivalent to\\
$KR(\omega^{(\br{\log n})})$ \cite[Th. 3]{Gri85}. Thus, for a typical $\omega$ in the measure-theoretic sense, $r_{\Gamma}(n)$ is linear for the corresponding $\ovBS$ and logarithmic for $G_{\omega}$ \cite{ZL70}.
\section{Backgrounds} \label{SecBack}

\subsection{The space of marked groups} \label{SubSecSpace}

The commonly used definition of the space of marked groups\footnote{A very similar topology was first considered in \cite[Final Remarks]{Grom81}. The general idea of topologizing sets of groups goes back to Mahlers and Chabauty \cite{Mah46,Chab50}. The interested reader should consult \cite{Har08} for a thorough account.}  is due to Grigorchuk who proved by a topological argument that his intermediate growth groups cannot be finitely presented \cite{Gri84}. The space of marked groups has then been used to prove both existence and abundance of groups with exotic properties \cite{Ste96, Cham00} and has turned to be a remarkably suited framework for the study of Sela's \emph{limit groups} \cite{ChGu05}. Isolated points were investigated in \cite{CGP07} and an isolated group is used by de Cornulier to answer a question of Gromov  concerning the existence of sofic groups which do not arise as limits of amenable groups. \cite{Cor09b}. Very little is known on its topological type \cite{Cor09a}. The box-counting dimension of the space of marked groups on $k$ generators is infinite if $k \ge 2$ \cite{Guy07}. We show that its Haudorff dimension is non-zero in Section \ref{SecDim}.

The free group on $k$ generators will be denoted by $\mathbb{F}_k$,
or $\Free(S)$ with $S = (s_1, \ldots, s_k)$, if we want to precise
the names of canonical generating elements. A \emph{marked group
on $k$ generators} is a pair $(G,S)$ where $G$ is a group and $S =
(s_1, \ldots, s_k)\in G^k$ is an ordered generating set of $G$. 
An \emph{isomorphism of marked groups} is an isomorphism which respects the markings. A marked
group $(G,S)$ is endowed with a canonical epimorphism $\phi:
\mathbb{F}_S \to G$, which induces an isomorphism of marked groups
between $\mathbb{F}_S/ \ker \phi$ and $G$. Hence a class of marked
groups for the relation of marked isomomorphism is represented by a unique quotient of $\mathbb{F}_S$. In
particular if a group is given by a presentation, this defines a
marking on it. The non-trivial elements of ${\mathcal R} := \ker \phi$
are called \emph{relations} of $(G,S)$. Given $w \in \mathbb{F}_k$
we will often write ``$w = 1$ in $G$" to
say that the image of $w$ in $G$ is trivial.

Let $w = x_1^{\varepsilon_1} \cdots x_n^{\varepsilon_n}$ be a
reduced word in $\mathbb{F}_S$ (with $x_i \in S$ and $\varepsilon_i
\in \{ \pm 1 \}$). The integer $n$ is called the \emph{length} of
$w$ and denoted $|w|$. If $(G,S)$ is a marked group on $k$
generators and $g \in G$, the \emph{length} of $g$ is
\begin{eqnarray*}
|g|_G & := & \min\{ n: g = s_1 \cdots s_n \text{ with } s_i \in S \sqcup S^{-1} \} \\
                    &  = & \min\{ |w| : w \in \mathbb{F}_S, \ \phi(w) = g  \} \ .
\end{eqnarray*}

Let ${\mathcal G}_k$ be the class of marked groups on $k$ generators 
up to marked isomorphism. Let us recall that the topology on ${\mathcal
G}_k$ comes from the following ultrametric distance: for $(G_1, S_1)
\neq (G_2, S_2) \in {\mathcal G}_k$ we set $d \big( (G_1, S_1), (G_2,
S_2) \big) := e^{-\lambda}$ where $\lambda$ is the length of a
shortest element of $\mathbb{F}_k$ which vanishes in one group and
not in the other one. One may also keep in mind the
following characterization of convergent sequences \cite[Pr. 1]{Sta06a}.

\begin{lemma}\label{lmecv}
Let $(G_n)$ be a sequence of marked groups in ${\mathcal
G}_k$. The sequence $(G_n)$ converges if and only if
for any $w \in {\mathbb F}_k$, we have either $w =1$ in $G_n$ for
$n$ large enough, or $w\neq 1$ in $G_n$ for $n$ large enough.
\end{lemma}

The free group $\Free_2=\Free(a,b)$, the Baumslag-Solitar groups $BS(p,q)$ and their limits $\ovBS$ are marked by their canonical generating pair $(a,b)$.
\subsection{The ring of $m$-adic integers} \label{SubSecZm}

Let $m \in \Z\setminus \{0\}$. Recall that the ring of $m$-adic integers
$\Z_m$\footnote{Note that $\Z_{-m}=\Z_{m}$ and that $\Z_m$ is the zero ring if $\vert m
\vert=1$.} is the projective limit in the category of topological
rings of the system
\[
\ldots \to \Z/m^h\Z \to \Z/m^{h-1}\Z \to \ldots \to \Z/m^2\Z \to
\Z/m\Z
\]
where the arrows are the canonical surjective homomorphisms. This
shows that $\Z_m$ is compact. This topology is compatible with the
ultrametric distance given, for $\xi \neq \eta$, by
\[
 |\xi-\eta|_m = |m|^{-\max\{k\in\N \, : \, \xi-\eta \, \in m^k\Z_m\}} \ .
\]
and $\Z$ embeds isomorphically and densely in $\Z_m$. To avoid ambiguity, we call elements of $\Z$ \emph{rational integers}. We only need the following easy facts about $m$-adic integers. Detailed proofs
can be found in the second-named author's Ph.D. thesis \cite[Appendix
C]{Sta05PhD}.
\begin{itemize}
\item There is a topological ring isomorphism from $\Z_m$ to $\bigoplus_{p | m}
\Z_p$, where $p$ ranges in the set of the prime divisors of $m$.
\item Any non-zero ideal can be uniquely written under the form $d \Z_m$ where $d$ is a positive rational integer whose prime divisors divide $m$. Moreover we have $\Z \cap d \Z_m=d\Z$. For any $n \in \Z$ and any $\xi \in \Z_m$, 
we can define the \emph{greatest common divisor} $\gcd(n,\xi)$ of $n$ and $\xi$
as the a unique positive rational integer $d$ such that $n$ and $\xi$ generate $d \Z_m$.
\item An $m$-adic integer $\xi$ is invertible if and only if it does not belong to any of the ideals $p \Z_m$ where $p$ is a prime divisor of $m$ in $\Z$; equivalently $\gcd(m,\xi)=1$.
\end{itemize}

Note finally that non-zero rational integers are never zero divisors in $\Z_m$.
Hence, we can consider the ring $\Z^{-1}\Z_m$ whose elements are fractions of
the form $\frac ab$ with $a\in\Z_m$ and $b\in \Z\setminus\{0\}$ and whose laws are the
classical ones for fractions.

\subsection{HNN extensions and tree actions} \label{SubSecHNN}
In this section we fix our notations for HNN extensions and collect several facts concerning their standard tree actions.
Let $G$ be a group, let $H,\,K$ be subgroups of $G$ and let $\tau: H \longrightarrow K$ be an isomorphism.
The \emph{HNN extension} of base $G$ whose stable letter $t$ conjugates $H$ to
$K$ via $\tau$, is the group
$$
\operatorname{HNN}(G,H,K,\tau)=\Pres{G,t}{tht^{-1}=\tau(h) \text{ for every } h
\in H}.
$$
Let $\Gamma=\operatorname{HNN}(G,H,K,\tau)$. We say that $\Gamma$ is an \emph{ascending} HNN extension if either $G=H$ or $G=K$. We say that $\Gamma$ is a \emph{degenerate} HNN extension if $G=H=K$, i.e. $\Gamma=G \rtimes \Z$ where $\langle t \rangle$ identifies with $\Z$. Note that a given group (e.g. $\Z \wr \Z$) may have two different HNN decompositions, one being degenerate whereas the other is not. 

\paragraph{The Normal Form Theorem}
A sequence $g_0, \,t^{\epsilon_1},g_1,\,\dots,\, t^{\epsilon_n},\,g_n$, with $n
\ge 0$, is said to be a \emph{reduced sequence} if there is no consecutive 
subsequence $t,\,g_i,\,t^{-1}$ with $g_i \in H$ or $t^{-1},\,g_i,\,t$ with $g_i
\in K$. Britton's lemma \cite[page 181]{LS77} asserts that the word
$g_0t^{\epsilon_1} \cdots t^{\epsilon_n} g_n$ has a non-trivial image $\gamma$
in $\Gamma$ if $g_0,\,t^{\epsilon_1},\,\dots,\,t^{\epsilon_n},\,g_n$ 
is reduced and $n \ge 1$. Such a word is called a reduced form for $\gamma$. Although $\gamma$ may have different reduced forms, the sequence $t^{\epsilon_1},\,\dots,\,t^{\epsilon_n}$ is uniquely determined by $\gamma$ and we call its length $n=|\gamma|_t$ the $t$-length of $\gamma$.
Fixing a set $T_H$ of representatives of right cosets of $H$ in $G$ and a set
$T_K$ of representatives of right cosets of $K$ in $G$ such that $1 \in T_H \cap
T_K$, the Normal Form Theorem \cite[Th. IV.2.1]{LS77} asserts that there is a
unique sequence $g_0,\,t^{\epsilon_1},\,\dots,\, t^{\epsilon_n},\,g_n$
representing $\gamma \in \Gamma$ with the following properties:
\begin{itemize}
\item $g_0$ is an arbitrary element of $G$,
\item If $\epsilon_i=1$, then $g_i \in T_H$,
\item If $\epsilon_i=-1$, then $g_i \in T_K$,
\item there is no consecutive subsequence $t^{\epsilon},\,1,\,t^{-\epsilon}$.
\end{itemize}
This sequence is the \emph{normal sequence of $\gamma$ with respect to $T_H$ and $T_K$} and we call the word $g_0 t^{\epsilon_1}\cdots  t^{\epsilon_n} g_n$ the \emph{normal form} of $\gamma$.

\paragraph {Action on the standard Bass-Serre tree and its boundary}
The \emph{standard Bass-Serre tree} of $\Gamma$ is the oriented tree $T$ with vertex set $V(T)=\Gamma/G$, with set of oriented edges $E_{+}(T)=\Gamma/H$ subject to the incidence relations $o(\gamma H)=\gamma G$ and $t(\gamma H)=\gamma t^{-1}G$.

Given a tree $T$, we define the \emph{boundary} $\partial T$ as the set of cofinal rays. 
The set $\partial T$ has a natural topology defined as follows. We
call a \emph{shadow} the boundary of any connected component of $T$
deprived of one of its edges. The family of shadows generates a
topology on $\partial T$ which is Hausdorff and totally
discontinuous. If $T$ is a countable non-linear tree, e.g.
the standard Bass-Serre tree of a non-degenerate HNN extension, then $\partial T$ is a perfect Baire
space \cite[Pr. 11 and 24.$ii'$]{HP09}. Note that every automorphism of $T$ induces an homeomorphism of $\partial T$.

 As $\Gamma$ acts without inversion on its Bass-Serre tree $T$, every element $\gamma \in \Gamma$, 
viewed as a tree automorphism, is either elliptic, i.e. $\gamma$ fixes some vertex of $T$, or hyperbolic, i.e. $\gamma$ 
fixes no vertex of $T$ but exactly two ends of $\partial T$ \cite{Ser77}. The action of $\Gamma$ on $\partial T$ has no fixed end if $\Gamma$ is non-ascending, exactly one fixed end if $\Gamma$ is non-degenerate and ascending, and exactly two fixed ends if $\Gamma$ is degenerate \cite[Pr. 24]{HP09}.
\section{HNN decomposition of the limits} \label{Sec1}

We fix $m \in \Z\setminus\{0\}$, $\xi \in \Z_m$ and $(\xi_n)$ a
sequence of rational integers such that $|\xi_n| \to \infty$ and
$\xi_n \to \xi$ in $\Z_m$. A natural HNN decomposition arises from the transitive action of $\ovBS$
on a tree constructed in \cite{GS08}. We recall this construction.

 We denote by $H_n$ (respectively $H_n^m$) the subgroup
of $BS(m,\xi_n)$ generated by $b$ (respectively $b^m$) and by $T_n$
the Bass-Serre tree of $BS(m,\xi_n)$. We set
$$
\begin{array}{ccccc}
Y    & = & \left(\prod\limits_{n\in\N} V(T_n)\right) / \sim   & = & \left(\prod\limits_{n\in\N} BS(m,\xi_n)/H_n \right) / \sim  \\
Y^m & = & \left(\prod\limits_{n\in\N} E_+(T_n)\right) / \sim & = &
\left(\prod\limits_{n\in\N} BS(m,\xi_n)/H_n^m\right) / \sim
\end{array}
$$
where $\sim$ is defined by $(x_n) \sim (y_n) \Longleftrightarrow
\exists n_0 \ \forall n \geqslant n_0: \ x_n = y_n$ in both cases.
We now define an oriented graph $X = X_{m,\xi}$ by
$$
\begin{array}{ccl}
V(X)                     & = & \{ x\in Y : \exists w \in \F \text{ such that } (x_n) \sim (wH_n) \}  \\
E_+(X)                   & = & \{ y\in Y^m : \exists w \in \F \text{ such that } (y_n) \sim (wH_n^m) \}  \\
o\big( (wH_n^m) \big) & = & (wH_n) = \big( o(wH_n^m) \big)  \\
t\big( (wH_n^m) \big) & = & (wa^{-1}H_n) =  \big( t(wH_n^m)
\big)
\end{array}
$$

 The graph $X_{m,\xi}$ is a tree and that the
obvious action of $\F$ on $X_{m,\xi}$ factorizes through the canonical
projection $\F \longrightarrow \ovBS$ \cite[Sec.3]{GS08}. Let $v_0=(H_n)$ and
$e_0=(H_n^m)$. We denote by $B$ (respectively $B_m$) the
stabilizer of $v_0$ (respectively $e_0$) in $\ovBS$. We set
$B_{\xi}=aB_m a^{-1}$. As the action of $\ovBS$ is clearly
transitive on vertices and geometrical edges, the group $\ovBS$ has a HNN
decomposition:

\begin{proposition} \label{PropHNNB}
The group $\ovBS$ is the HNN extension of base group $B$, stable
letter $a$ and conjugated subgroups $B_m$ and $B_{\xi}=aB_m a^{-1}$.
\end{proposition}
We define the inner automorphism $\tau_a:\gamma \mapsto a \gamma a^{-1}$ of
$\ovBS$ and we denote by $\text{HNN}(B,B_m,B_{\xi},\tau)$ the
previous HNN decomposition where $\tau$ is the isomorphism from $B_m$ to $B_{\xi}$ induced by $\tau_a$.
\begin{proof}
This follows from results in \cite{Ser77}. See in particular Section
I.5 and Remark 1 after Theorem 7 in Section I.4.
\end{proof}

The following lemma is an immediate consequence of the definition of the action:

\begin{lemma} \label{LemCarB}
Let $w \in \F$ and let $\gamma$ be its image in $\ovBS$. Then $\gamma
\in B$ (respectively $\gamma \in B_m$) if and only if
$w=b^{\lambda_n}$ (respectively $b^{m\lambda_n}$) in $BS(m,\xi_n)$
for all $n$ large enough.
\end{lemma}

As a result, $B$ is abelian. Let $\Z \br{X}$ be the ring
of polynomials in the variable $X$ with integer coefficients. A
thorough study of the sequence $(\lambda_n)$ of Lemma \ref{LemCarB}
enables us to embed $B$ isomorphically into the additive group of $\Z \br{X}$:

\begin{proposition} \label{PropEmbedZX}
\noindent
\begin{itemize}
\item Let $w \in \F$ with image $\gamma \in B$. Then there is a unique
polynomial $P_{\gamma}(X) \in \Z \br{X}$ such that
$w=b^{P_{\gamma}(\xi_n/m)}$ in $BS(m,\xi_n)$ for all $n$ large
enough.
\item The map $q:\gamma \mapsto P_{\gamma}(X)$ defines an injective homomorphism
from $B$ into $\Z \br{X}$.
\end{itemize}
\end{proposition}
The construction of the homomorphism $q$ relies on a sequence of polynomials built up from $m$ and $\xi$ 
in a non obvious way. Its definition is the subject of the next section. The
proof of Proposition \ref{PropEmbedZX} is therefore postponed to Section
\ref{SubSecStab} (see Proposition \ref{PropBandBm}). 

\subsection{The functions $r_i$ and $s_i$} \label{SubSecRiSi}
In this section we define functions $r_i$ and $s_i$ on $\Z$ which describe how the $b$ exponents of a
given word $w\in \F$ behave when we reduce it in $BS(m,n)$ for $n$
ranging in a given class modulo $m^h$. In \cite{Sta06b,GS08}, these functions were decisive to describe converging sequences of Baumslag-Solitar groups.  It turns out that they extend continously to $\Z_{m}$ and that $\xi$, and hence $\ovBS$, is uniquely determined by $m$ and the sequence $(r_i(\xi))$.
\begin{definition} \label{DefiSi}
We define the functions $r_0,\,r_1,\,\dots,\,s_0,\,s_1,\dots$ on $\Z$
depending on a parameter $ m \in  \Z \setminus \{0\}$ by the inductive formulas:
\begin{equation} \label{eqrec1}
r_0(n)=0,\,s_0(n)=1;
\end{equation}
\begin{equation} \label{eqrec2}
n s_{i-1}(n)= m s_i(n) +r_i(n)  \mbox{ with } r_i(n) \in
\{0,\,\dots,\,| m|-1\} \text{ for every } i \ge 1.
\end{equation}
\end{definition}

\begin{proposition} \label{PropRiSi}
Let $n,n'\in \Z$ such that $\gcd(n, m)=\gcd(n', m)= d $. Let $h\in
\N \setminus \{0\}$ and $\mh= m/ d $.
\begin{itemize}
\item[$(1)$] Assume that $n \equiv n' \, ( \operatorname{mod}\mh^h d \Z)$. Then the
two following hold:
\begin{itemize}
\item[$(i)$] $r_i(n)=r_i(n')$ for $i=1,\,\dots,\,h$;
\item[$(ii)$] $s_i(n) \equiv s_i(n') \, (\,
\operatorname{mod} \mh^{h-i}\Z) \text{ for } i=1,\,\dots,\,h.$
\end{itemize}
\item[$(2)$] If $r_i(n)=r_i(n')$ for
$i=1,\dots,h$ then $n \equiv n' \, ( \operatorname{mod}\mh^h d \Z)$.
\end{itemize}
\end{proposition}

\begin{sproof}{Proof of Proposition \ref{PropRiSi}}

Proof of $(1)$. We show by induction on $i$ that $$r_i(n)=r_i(n'), \,s_i(n) \equiv s_i(n') \, (\operatorname{mod}\mh^{h-i}\Z) \text{ and }s_i(n)
\cdot n\equiv s_i(n') \cdot n'\, (\operatorname{mod}\mh^{h-i} d \Z)$$ for
$i=0,\,1,\,\dots,\,h$. The case $i=0$ is obvious. Assume that
$r_{i-1}(n)=r_{i-1}(n'),\,s_{i-1}(n) \equiv s_{i-1}(n') \, (\text{mod
}\mh^{h-i+1}\Z)$ and $s_{i-1}(n) \cdot n \equiv s_{i-1}(n') \cdot n'
\,(\operatorname{mod}\mh^{h-i+1} d \Z)$ for some $1 \le i \le h$. By
construction
$$\left\{\begin{array}{c} s_{i-1}(n) \cdot
n=s_i(n)\cdot  m+r_i(n)\\ \text{ and } \\
 s_{i-1}(n') \cdot
n'=s_i(n')\cdot  m+r_i(n').
\end{array}\right.$$
We deduce from the induction hypothesis that $r_{i}(n)  \equiv
r'_{i}(n') \,(\operatorname{mod}  m\Z)$ and hence $r_i(n)=r_i(n')$. As a
result
$$
s_i(n)-s_i(n')=\frac{(s_{i-1}(n) \cdot n -s_{i-1}(n') \cdot
n')}{ m},
$$
from which we deduce that
\begin{equation} \label{EqSi} s_i(n)
\equiv s_i(n') \,(\operatorname{mod} \mh^{h-i}\Z)
\end{equation}
still by using the induction hypothesis. Look at the right member of
the relation
\[
  s_i(n)\cdot n-s_i(n') \cdot n'= (s_i(n)-s_i(n'))\cdot n + s_i(n') \cdot(n-n') \ .
\]
The first term is a multiple of $\mh^{h-i} d $ because of equation
(\ref{EqSi}) and the fact that $ d $ divides $n$. The second one is
a multiple of $\mh^{h-i} d $ because of hypothesis $(1)$. Therefore,
we get $s_i(n) \cdot n \equiv s_i(n') \cdot n' \,(\operatorname{mod}
\mh^{h-i} d \Z)$ as desired.

Proof of $(2)$. By construction, we have:
$$\left\{\begin{array}{c} s_{i-1}(n) \cdot
n=s_i(n) \cdot m+r_i(n)\\ \text{ and } \\
 s_{i-1}(n') \cdot
n'=s_i(n')\cdot  m+r_i(n')
\end{array}\right.$$ for all $i\in\{1,\dots,h\}$. By hypothesis $(2)$,
we get $s_i(n)-s_i(n')=(s_{i-1}(n)\cdot n-s_{i-1}(n') \cdot n')/m$,
which we write into the following form:
\begin{equation} \label{EqQuotientSi}
s_i(n)-s_i(n')=\frac{(s_{i-1}(n)-s_{i-1}(n')) \cdot
\nh}{\mh}+\frac{s_{i-1}(\nh')(\nh-\nh')}{\mh}.
\end{equation}
where $\nh=n/ d $ and $\nh'=n'/ d $. We show by induction on $k$
that $\nh \equiv \nh' \,(\operatorname{mod} \mh^k\Z)$ for $k=1,\dots,h$.
The case $k=1$ follows directly from equation (\ref{EqQuotientSi})
for $i=1$ (recall that $s_0(n)=1=s_0(n')$). For the inductive step,
we assume that $2 \le k \le h$ et $\nh \equiv \nh'\,(\operatorname{mod}
\mh^{k-1}\Z)$, which implies that the second term in the right
member of equation (\ref{EqQuotientSi}) is a multiple of $\mh^{k-2}$
(it is an integer in particular). We then proceed in $k$ steps using
equation (\ref{EqQuotientSi}) and the fact that $\mh$ and $\nh$ are
coprime:
\begin{itemize}
\item[-]for $i=k$ we get $s_{k-1}(n) \equiv s_{k-1}(n')\,(\operatorname{mod}
\mh\Z)$;
\item[-]for $i=k-1$ (if $k\ge3$), we get $s_{k-2}(n) \equiv s_{k-2}(n')\,(\operatorname{mod}
\mh^2\Z)$;
\item[-]so on so forth with $i=k-2,\dots,2$, we get the sequence of congruences $$s_{k-3}(n) \equiv s_{k-3}(n')\,(\operatorname{mod}
\mh^3\Z),\dots,s_1(n)\equiv s_1(n') \,(\operatorname{mod} \mh^{k-1}\Z);$$
\item[-]for $i=1$, we get $\nh \equiv \nh'\,(\operatorname{mod}
\mh^k\Z)$.
\end{itemize}
Finally, we obtain $n\equiv n'\, (\operatorname{mod} \mh^k d \Z)$ and the
proof is then complete.
\end{sproof}

\begin{proposition} \label{PropRiSiCont}
\begin{itemize}
\item[$(i)$]

The functions $r_i$ and $s_i$ admit unique continuous extensions
\[
 r_i:\Z_{ m} \to \{0,\ldots, | m|-1\},\quad  s_i:\Z_{ m} \to \Z_{ m}
\]
such that
\begin{equation} \label{EqRecXi1} \tag{$1_{\xi}$}
r_0(\xi)=0,\,s_0(\xi)=1;
\end{equation}
\begin{equation} \label{EqRecXi2} \tag{$2_{\xi}$}
\xi s_{i-1}(\xi)= m s_i(\xi) +r_i(\xi) \mbox{ with } r_i(\xi) \in
\{0,\dots,| m|-1\} \text{ for every } i \ge 1.
\end{equation}

Let $\xi,\xi' \in \Z_{ m}$ such that
$\gcd(\xi, m)=\gcd(\xi', m)= d $ and let $h \in  \N \setminus \{0\}$. Setting
$\hat m = m/d$, the two
following are equivalent:
\begin{itemize}
\item[$(1)$] $\xi \equiv  \xi' \, ( \operatorname{mod}\mh^h d \Z_{ m})$ holds;
\item[$(2)$] $r_i(\xi)=r_i(\xi')$ for $i=1,\dots,h$.
\end{itemize}

\item[$(ii)$] Let $\xi,\xi' \in \Z_{ m}$.
The two
following are equivalent:
\begin{itemize}
\item[$(1)$] There is some $d \in \N$ such that $\gcd(\xi, m)=\gcd(\xi', m)= d$ 
and $\pi(\xi/d)=\pi(\xi'/d)$ where $\pi:\Z_{m} \longrightarrow \Z_{m/d}$ is the
canonical ring homomophism,
\item[$(2)$] $r_i(\xi)=r_i(\xi')$ for every $i$.
\end{itemize}

\item[$(iii)$] The map $\xi \mapsto (r_i(\xi))_{i \ge 1}$ defines a
 homeomorphism from $\Z_{m}^{\times}$ onto $(\Z/ m\Z)^{\times} \times (\Z/ m\Z)^{\N}$
 endowed with its product topology. (Here we identify $\Z/ m\Z$ with the set $\{0,\dots, | m|-1\}$ and $(\Z/ m\Z)^{\times}$
 with the set of integers $k\in \{0,\dots, |m|-1\}$ coprime with $ m$.)
\end{itemize}

\end{proposition}

\begin{proof}

$(i)$ Proposition \ref{PropRiSi} shows that the functions
$r_i,\,s_i$ are uniformly continuous with respect to the $ m$-adic
topology on $\Z$ and that $r_i$ is moreover constant on $ m$-adic
balls of radius $| m|^{-i}$. The existence and uniqueness of the
extensions to $\Z_ m$ follow. The inductive formulas of Definition
\ref{DefiSi} extend continuously on $\Z_{ m}$, which gives
$(1_{\xi})$ and $(2_{\xi})$. Let us now pick $n,n'\in \Z$ such that
$n\equiv \xi$ and $n'\equiv\xi'$ $(\operatorname{mod}  m^h\Z_ m)$. Then,
one has $n\equiv n' \, (\operatorname{mod} \mh^h d )\Z$ if and only if
$\xi\equiv \xi' \, (\operatorname{mod} \mh^h d \Z_ m)$. As $r_i$ is
constant on $ m$-adic balls of radius $| m|^{-i}$ we have
$r_i(n)=r_i(\xi)$ and $r_i(n')=r_i(\xi')$ for $i=1,\ldots,h$. The
equivalence between (1) and (2) now follows readily from Proposition
\ref{PropRiSi}.

$(ii)$ As $\gcd(m,\xi)=\gcd(m,r_1(\xi))$ for every $\xi \in \Z_{m}$ by $(2_{\xi})$, the result immediatly follows from $(i)$.

$(iii)$ Let $\xi \in \Z_{ m}^{\times}$. By ($2_{\xi}$), we have $r_1(\xi) \in
(\Z/ m\Z)^{\times}$. For any $\xi,\xi' \in \Z_{ m}^{\times}$, we
have $d=\gcd(\xi, m)=\gcd(\xi', m)=1$. We deduce from $(2)$ that
the map $R:\xi \mapsto (r_i(\xi))_{i\ge1}$ defines a continuous
embedding from $\Z_{ m}^{\times}$ into $P:=(\Z/ m\Z)^{\times}
\times (\Z/ m\Z)^{\N}$. As $\Z_{ m}^{\times}$ and $P$ are compact,
it is suffices to show that $R$ has a dense image. Let $E_h$ be the
set of integers $k\in \{0,\dots, | m|^h-1\}$ coprime with $ m$.
Consider the map $R_h:\xi \mapsto (r_i(\xi))_{1 \le i \le h}$. This
is an injective map from $E_h$ into $P_h:=(\Z/ m\Z)^{\times} \times
(\Z/ m\Z)^{h-1}$ by the equivalence of $(1)$ and $(2)$. As $E_h$ and
$P_h$ have both cardinality $\varphi( m)|m|^{h-1}$, where $\varphi$
denotes the Euler function, the map $R_h$ is a bijection. Hence $R$
maps the set of integers coprime with $|m|$ onto a dense subset of
$P$.
\end{proof}

The following proposition shows that the restriction of $s_i$ to a suitable congruence classe is a
polynomial in $n$ with coefficients in $\Q$.

\begin{proposition}\label{DefiPropPi}
Let $\xi \in\Z_{ m}$. We define recursively $P_h(X)=P_{h,\xi}(X)$ by
$$P_0(X)= m \text{ and } P_h(X)=XP_{h-1}(X)-r_h(\xi) \text{
for } h \ge 1.$$
Then, we have:
\begin{itemize}
\item[$(i)$]
$P_{h}(X)= m X^{h}-r_1(\xi)X^{h-1}-\dots-r_h(\xi) \in\Z \br{X}$.
\item[$(ii)$]$s_h(n)=\frac{1}{ m}P_{h}(\frac{n}{ m})$ for all $h \ge 0$ and
all $n\in \Z$ such that $n \equiv \xi \,(\operatorname{mod}  \mh^h  d \Z_m)$.
\end{itemize}
\end{proposition}

\begin{proof}
$(i)$ The proof is an obvious induction on $h$.\\
$(ii)$ First, observe that $r_i(n)=r_i(\xi)$ for all $i \le h$ and
all $n\in \Z$ such that $n \equiv \xi \,(\operatorname{mod} \mh^h d \Z_m)$ 
by Proposition \ref{PropRiSiCont}. An easy
induction on $h$ using the definitions of $s_i(n)$ and $P_i(X)$ gives the
conclusion.
\end{proof}

\subsection{Stabilizers} \label{SubSecStab}
This section is devoted to the study of the stabilizers $B,B_m$ and
$B_{\xi}$.

Let $b_0=b,\,b_1=ab^ma^{-1},\,b_{i}=ab_{i-1}b_0^{-r_{i-1}(\xi)}a^{-1}$
for $i \ge 2$. The following lemma shows that $b_i$ defines an
element of $B$ for every $i$.

\begin{lemma} \label{LemBi}
Let $i \ge 1$. We have $b_i=b^{
\frac{\xi_n}{m}P_{i-1}(\frac{\xi_n}{m})}\text{ in } BS(m,\xi_n)
\text{ for all }n \text{ large enough.}$
\end{lemma}

\begin{proof}
We show by induction on $i$ that: \begin{equation} \label{EqBi}
\text{ for every } i\ge 1 \text{ we have } b_i=b^{ \xi_n
s_{i-1}(\xi_n)}\text{ in } BS(m,\xi_n) \text{ for all }n \text{
large enough.}\end{equation} Since $b_1=ab^ma^{-1}=b^{\xi_n}$ in
$BS(m,\xi_n)$ and $s_0(\xi_n)=1$ for all $n$, (\ref{EqBi}) holds if
$i=1$. Assume (\ref{EqBi}) holds for some $i \ge 1$. By the
induction hypothesis we have $b_{i+1}=ab_ib^{-r_i(\xi)}a^{-1}=ab^{
\xi_n s_{i-1}(\xi_n)-r_i(\xi)}a^{-1}$ for all $n$ large enough.
Recall that $\xi_n$ tends to $\xi$ in $\Z_{m}$ and hence
$r_i(\xi_n)=r_i(\xi)$ for all $n$ large enough by Proposition
\ref{PropRiSiCont}. By Definition \ref{DefiSi}, we obtain
$b_{i+1}=b^{ \xi_n s_{i}(\xi_n)}\text{ in } BS(m,\xi_n) \text{ for
all }n \text{ large enough.}$  Since $ \xi_n s_{i}(\xi_n)=
\frac{\xi_n}{m}P_{i}(\frac{\xi_n}{m})$ for all $n$ large enough by
Proposition \ref{DefiPropPi}, the proof is then complete.
\end{proof}
We can generalize the previous lemma by assigning to every $\gamma \in B$
a polynomial with integer coefficients.

\begin{proposition} \label{PropBandBm}
\noindent
\begin{itemize}
\item[$(i)$] Let $w \in \F$ with image $\gamma \in B$. Then there is a unique
polynomial $P_{\gamma}(X) \in \Z \br{X}$ independent of $w$ such
that $w=b^{P_{\gamma}(\xi_n/m)}$ in $BS(m,\xi_n)$ for all $n$ large
enough.
\item[$(ii)$] The map $q:\gamma \mapsto P_{\gamma}(X)$ defines an injective homomorphism
from $B$ into $\Z \br{X}$. The abelian group $\mathfrak{B}=q(B)$ is
freely generated by $\{1,\,XP_0(X),\,
XP_1(X),\,X P_2(X),\,\dots\}$. Hence $B$ is freely generated by the set
$\{b_0,\,b_1,\,b_2,\,b_3,\,\dots\}$.
\item[$(iii)$] Let $w \in \F$ with image $\gamma \in B$.
 Then $\gamma$ belongs to $B_m$ if and only if $$P_{\gamma}(X)=k_0+k_1m X+k_2 XP_1(X)+\dots+k_{t}
XP_{t-1}(X)$$ with $k_0+k_1r_1(\xi)+k_2 r_2(\xi)+\dots+k_tr_t(\xi)
\equiv 0 \,(\operatorname{mod}m\Z).$ Moreover the abelian group
$\mathfrak{B}_m=q(B_m)$ is freely generated by $P_0(X)=m,\, P_1(X)=XP_0(X)
-r_1(\xi),\, P_2(X)=X P_1(X)-r_2(\xi),\,\dots$. Hence $B_m$ is
freely generated by $\{b_0^m,\,b_1b_0^{-r_1(\xi)},\,b_2
b_0^{-r_2(\xi)},\,\dots\}$.

\item[$(iv)$] The abelian group $\mathfrak{B}_{\xi}=q(B_{\xi})$ is freely
generated by $\{XP_0(X),\, XP_1(X),\,XP_2(X),\,\dots\}$. The abelian group $B_{\xi}$ is
freely generated by $\{b_1,\,b_2,\,b_3,\dots\,\}$.

\item[$(v)$] For any $\gamma \in B_m$, we have
$q(a\gamma a^{-1})=X P_{\gamma}(X)$. Moreover the map $a \mapsto
(0,1),\,b \mapsto (1,0)$ induces a surjective homomorphism $q_{m,\xi}:\ovBS
\longrightarrow\Z \wr \Z=\Z \br{X^{\pm 1}} \rtimes \Z$ whose
restriction to $B$ coincides with $q$.

\item[$(vi)$] Let $d=\gcd(m,\xi),\, \mh=m/d$ and let $\pi: \Z \longrightarrow
\Z_{\mh}$ be the canonical map. The maps $\chi:\gamma \mapsto
P_{\gamma}(\xi/m)$ and $\hat{\chi}=\pi \circ \chi$
define homomorphisms from $B$ to $\Z_{m}$ and $\Z_{\mh}$ respectively. Their
kernels satisfy:
 $\ker \chi \subset \bigcap_{i \ge 0} a^{-i}Ba^i \subset \ker\hat{\chi}$. 
Moreover, if $\bigcap_{i \ge 0} a^{-i}Ba^i$ is non-trivial, then it contains
some $\gamma \in B_m \setminus B_{\xi}$.

\item[$(vii)$] Let $\mathfrak{C}$ be the image of $\mathfrak{B}_{\xi}$ by the map $\iota:P(X) \mapsto \frac{X-1}{X}P(X)$. Then
$\mathfrak{C}$ is a subgroup of $\mathfrak{B}$ and $\mathfrak{B}/\mathfrak{C}$ is an infinite cyclic group generated by the image of $1$. 

\end{itemize}
\end{proposition}


\begin{sproof}{Proof of Proposition \ref{PropBandBm}}
$(i)$ The uniqueness of $P_{\gamma}(X)$ follows from the fact that a non-zero
polynomial with coefficients in $\Q$ has only finitely many zeros in $\Q$. Using the fact that
$BS(m,\xi_n)$ converges to $\ovBS$, we also deduce that
$P_{\gamma}(X)$ is independent of $w$.  To show the existence of
$P_{\gamma}(X)$ we write $w=b^{e_0} a^{\varepsilon_1} b^{e_1} \cdots
a^{\varepsilon_h} b^{e_h}$ with $\varepsilon_j =\pm 1$ for
$j=1,\dots,h$. By Lemma \ref{LemCarB}, $w$ reduces to a power of $b$
in $BS(m,\xi_n)$ for all $n$ large enough. Consequently we have
$\sigma_a(w)=0$ and hence $h=2t$ for some $t \in \N$. By \cite[Lemma
2.6 with $d=1$]{GS08} there exist $k_0,\dots,k_t \in \Z$ depending
only on $w$ and $\xi$ such that $w =b^{\alpha(n)}$ in $BS(m,\xi_n)$
for all $n$ such that $\xi_n \equiv \xi \,(\operatorname{mod}m^t \Z_{m})$
with $\vert \xi_n \vert$ large enough and
\begin{equation} \label{EqAlpha} \alpha(n)=k_0+k_1\xi_n +k_2
s_1(\xi_n)\xi_n +\dots+k_{t} s_{t-1}(\xi_n) \xi_n ,\end{equation}
the latter equation being Formula (*) in the proof of \cite[Lemma
2.6]{GS08}. Since $\xi_n$ tends to $\xi$ in $\Z_{m}$ and $\vert
\xi_n \vert$ tends to infinity as $n$ goes to infinity, the equality
(\ref{EqAlpha}) holds for all $n$ large enough. Using Proposition
\ref{DefiPropPi}, we can write
$\alpha(n)=P_{\gamma}(\frac{\xi_n}{m})$ with
\begin{equation} \label{EqPgamma}
P_{\gamma}(X)=k_0+k_1XP_0(X)+k_2XP_1(X)+\dots+k_{t} XP_{t-1}(X).
\end{equation}

$(ii)$ The map $q:\gamma \mapsto P_{\gamma}(X)$ trivially induces a
homomorphism from  $B$ to $\Z \br{X}$ such that $q(b)=1$. Let $w \in \F$ with image
$\gamma \in B$. If $P_{\gamma}(X)$ is the zero polynomial then $w$
is trivial in $BS(m,\xi_n)$ for all $n$ large enough and hence it is
trivial in $\ovBS$. Thus $q$ is injective. It is immediate from
(\ref{EqPgamma}) that $q(B)$ is a subgroup of the free abelian group
with basis $\{1, XP_0(X),XP_1(X),\dots\}$. It follows from Lemma
\ref{LemBi} that $q(b_i) =XP_{i-1}(X)$ for all $i\geq 1$, so that
$q(B)$ coincides with the latter group.

$(iii)$ Let $w \in \F$ with image $\gamma \in B$. By $(i)$
we can write $w=b^{P_{\gamma}(\xi_n/m)}$ in $BS(m,\xi_n)$ for all
$n$ large enough. Since $\xi_n s_{i-1}(\xi_n) \equiv r_i(\xi) \,
(\operatorname{mod} m\Z)$ for every $i \ge 1$ and for all $n$ large
enough, we deduce from (\ref{EqPgamma}) that $P_{\gamma}(\xi_n/m)
\equiv k_0+k_1r_1(\xi)+k_2 r_2(\xi)+\dots+k_tr_t(\xi) \,(\text{mod
}m\Z)$ for all $n$ large enough. By Lemma \ref{LemCarB}, we have:
$\gamma \in B_m$ if and only if $\alpha(n) \equiv 0 \,(\operatorname{mod}
m\Z)$. This proves the first claim. We easily deduce that
$\{m,\,XP_0(X)-r_1(\xi),\,X P_1(X)-r_2(\xi),\,\dots\}$ freely generates
$\mathfrak{B}_m$. Since $q(b_0^m)=m$ and $q(b_i
b_0^{-r_i(\xi)})=XP_{i-1}(X)-r_i(\xi)$ the set $\{b_0^m,\,b_1
b_0^{-r_1(\xi)},\,b_2 b_0^{-r_2(\xi)},\,\dots\}$ freely generates $B_m$.

$(iv)$ Since $B_{\xi}=aB_ma^{-1}$, we deduce from $(iii)$ and the
definition of $b_i$ that $\{b_1,b_2,\dots\}$ generates $B_{\xi}$. As the elements
$q(b_i)=XP_{i-1}(X)$ ($i \ge 1$) freely generate
$\mathfrak{B}_{\xi}$, we deduce that $\{b_1,\,b_2,\,\dots\}$ freely generates $B_{\xi}$.

$(v)$ Let $w \in \F$ with image $\gamma \in B_m$. We can write
$w=b^{P_{\gamma}(\xi_n/m)}$ in $BS(m,\xi_n)$ for all $n$ large
enough with $P_{\gamma}(\xi_n/m) \equiv 0 \,(\operatorname{mod}m\Z)$. Thus
$a\gamma a^{-1}=b^{(\xi_n/m) P_{\gamma}(\xi_n/m)}$ in $BS(m,\xi_n)$
for all $n$ large enough. Hence $q(a\gamma a^{-1})=XP_{\gamma}(X)$
by definition of $q$.

The map $q_{m,\xi}$ is well-defined homomorphism by \cite[Th. 3.12]{GS08}.
Using the first part of $(v)$, we easily deduce by induction that $q(b_i)=q_{m,\xi}(b_i)$ 
for every $i \ge 0$, which completes the proof.

$(vi)$ The map $\gamma \mapsto P_{\gamma}(\xi/m)$ is a well-defined
homomorphism from $B$ to $\Z^{-1} \Z_{m}$. By Proposition
\ref{PropRiSiCont} we have $s_i(\xi) \in \Z_{m}$ for every $i$. By
Proposition \ref{DefiPropPi}, we have
$\frac{1}{m}P_i(\frac{\xi_n}{m})=s_i(\xi_n) \in \Z$ for every $n$ large
enough and hence $\frac{1}{m}P_i(\frac{\xi}{m})=s_i(\xi) \in \Z_{m}$
by Proposition \ref{PropRiSiCont}. Now it follows from $(ii)$ that $P_{\gamma}(\xi/m)$ belongs to $\Z_{m}$ for every
$\gamma \in B$. 

Let $\gamma \in \ker \chi$ and let $i \ge 0$. For all $n$ large enough $m^i$
divides $P_{\gamma}(\xi_n/m)$. We deduce from $(i)$ and Lemma
\ref{LemCarB} that $a^i \gamma a^{-i} \in B$. Thus $\gamma \in \bigcap_{i \ge
0}a^{-i} B a^i$. 

Let $\eta=\pi(\xi/d)$. By $(v)$ we have $\hat{\chi}(a\gamma
a^{-1})=\frac{\eta}{\mh} \hat{\chi}(\gamma) \in \Z_{\mh}$ for every $\gamma \in
B \cap a^{-1} B a$. Consider now $\gamma \in B$ such that $a^i \gamma a^{-i} \in
B$ for every $i$. As $\gcd(\mh,\eta)=1$,  $\mh^i$ divides
$\hat\chi(\gamma)$ for every $i$. Therefore
we have $\gamma \in \ker\hat\chi$.

Assume that $B'=\bigcap_{i \ge 0} a^{-i}B a^{i}$ is non-trivial. Clearly $B' \subset
B_m$ and any $\gamma \in B'$ such that $P_{\gamma}(X)$ has minimal degree does
not belong to $B_{\xi}$ by $(v)$.

$(vii)$ It follows from the definition of $P_i(X)$ that multiplication by
$\frac{X-1}{X}$ maps $XP_0(X)$ to $-m+XP_0(X)$ and $XP_i(X)$ to
$r_{i}(\xi)-XP_{i-1}(X)+XP_i(X)$ for every $i \ge 1$. Therefore $\mathfrak{C}$
is a subgroup of $\mathfrak{B}$. A straightforward induction on $i$ shows that
the image of $XP_i(X)$ in $\mathfrak{B}/\mathfrak{C}$ lies inside the cyclic
subgroup generated by the image of $1$. As $\mathfrak{B}_{\xi}$ does not contain
any constant polynomial, neither does $\mathfrak{C}$. Therefore
$\mathfrak{B}/\mathfrak{C}$ is infinite. 
\end{sproof}

Now, we can give a simple description of the HNN structure of
$\ovBS$. Let $\widetilde{BS}(m,\xi)$ be the HNN extension of basis $E$
with conjugated subgroups $E_{m,\xi}$ and $E_1$ stable letter $a$, where
$E, E_m, E_\xi$ are free abelian groups of countable rank, namely
\begin{eqnarray*}
 E & = & \Z e_0 \oplus \Z e_1 \oplus \Z e_2 \oplus \cdots \\
 E_{m,\xi} & = & \Z me_0 \oplus \Z(e_1 - r_1(\xi)e_0) \oplus \Z(e_2 - r_2(\xi)e_0) \cdots \\
 E_1 & = & \Z e_1 \oplus \Z e_2 \oplus \cdots
\end{eqnarray*}
and conjugacy from $E_{m,\xi}$ to $E_1$ is defined by $a(me_0)a^{-1} =
e_1$ and $a(e_i-r_i(\xi)e_0)a^{-1} = e_{i+1}$. We denote by $\phi$ the isomorphism from $E_{m,\xi}$ to $E_1$ induced by $\tau_a$.

\begin{corollary} \label{CorHNN2}
The map defined by $f(a)=a$ and $f(b)=e_0$ induces an isomorphism from
$\ovBS$ to $\tBS{m}{\xi}$. The inverse map is determined by $f^{-1}(a)=a$ and $f^{-1}(e_i) = b_i$ for every $i \ge 0$.
\end{corollary}

\begin{proof}
 It is a straightforward corollary of Propositions
\ref{PropHNNB} and \ref{PropBandBm}.
\end{proof}

We have then three different notations for the same group, namely
\begin{itemize}
\item $\ovBS=\text{HNN}(B,B_m,B_{\xi},\tau)$ where $B$ is generated by $b=b_0,\,b_1,\,b_2,\,\dots$,
\item $\tBS{m}{\xi} = \text{HNN}(E,E_{m,\xi},E_1,\phi)$ where $E$ is generated by $e_0,\,e_1,\,e_2,\,\dots$,
\item $\mathfrak{BS}(m,\xi)=\text{HNN}(\mathfrak{B},\mathfrak{B}_m,\mathfrak{B}_{\xi},\theta)$ where $\mathfrak{B}$ is generated by $1,\,XP_0(X),\,XP_1(X),\,\dots$ and $\theta$ is defined in the obvious way.
\end{itemize}

Our favoured HNN extension is $\tBS{m}{\xi}$ as it considerably eases off notations when different parameters $m$ and $\xi$ are considered (Sec. \ref{SecHomo} and Sec. \ref{SecDim}). Nevertheless $\ovBS$ is usefull when we still need to see this group as limit in the space of marked groups (Sec. \ref{SecEqNoe})
 and the HNN extension $\mathfrak{BS}(m,\xi)$ involving polynomials 
is advantageous to study relations in the group (Sec. \ref{SecHomo}). 
\section{An infinite presentation built up from $\xi$} \label{SecPres}

In this section, we use of the HNN decomposition of $\BS$ to
give an infinite group presentation which depends explicitely on $m$
and the sequence $(r_i(\xi))$. A group presentation $\Pres{X}{R}$
is said to be \emph{minimal} if the kernel of the natural
homomorphism $\Pres{X}{R'} \longrightarrow \Pres{X}{R}$ is
non-trivial for all $R'\varsubsetneq R$.

\begin{theorem} \label{ThPres}
 The group $\ovBS$ admits the following infinite presentation:
$$
\Pres{a,b}{\br{b,b_i}=1,\,i\ge 1}
$$
with $b_1=ab^ma^{-1}$ and $b_{i}=ab_{i-1}b^{-r_{i-1}(\xi)}a^{-1}$
for every $i \ge 2$.

In particular, we have $\overline{BS}(m,0)=\Pres{a,b}{\br{b,a^ib^ma^{-i}}=1,\, i \ge 1}$. The
latter presentation is moreover minimal.
\end{theorem}

In the case $\xi \neq 0$, we previously showed \cite[Th. 4.1]{GS08}
that $\ovBS$ cannot be finitely presented. The following consequence
is therefore immediate.
\begin{corollary}\label{CorPresInf}
 No group $\BS$ is finitely presented.
\end{corollary}

\begin{sproof}{Proof of Theorem \ref{ThPres}}

We divide the proof into two parts:
\item[$(1)$] We show that
$\ovBS=\Pres{a,b}{\br{b,b_i}=1,\,i \ge 1}$ using Propositions
\ref{PropHNNB} and \ref{PropBandBm}.
\item[$(2)$] We show that this presentation is minimal in the case $\xi=0$ using basic facts on graph groups.

\item[Proof of $(1)$.]
 By Proposition \ref{PropBandBm} we can define an isomorphism $\phi:B_m \longrightarrow B_{\xi}$
by $\phi(b_0^m)=b_1,\,\phi(b_ib_0^{-dr_{i}(\xi)})=b_{i+1}$ for $i \ge
1$. Using the HNN decomposition of Proposition \ref{PropHNNB}, we deduce
that $\ovBS$ admits the presentation $\Pres{a,B}{a \gamma
a^{-1}=\phi(\gamma),\, \gamma \in B_m}$. Since $a$ and $b$ generate
$\ovBS$ and $B$ is a free abelian group (Proposition
\ref{PropBandBm}), we deduce that
\begin{equation}\label{EqPres}\ovBS=\Pres{a,b}{\br{b_i,b_j}=1,\, 0 \le i<j}\end{equation}(Note that the
relations $ab_0^ma^{-1}=\phi(b_0^m)$ and
$ab_ib_0^{-r_i(\xi)}a^{-1}=\phi(b_ib_0^{-r_i(\xi)})$ are satisfied
by definition of $b_i$.) We show by induction on $i\ge 0$ the
following claim: any relation $\br{b_i,b_j}=1$ with $0 \le i<j$ is a
consequence of relations $\br{b_0,b_k}=1$ with $1 \le k\leq j$.
 The case $i=0$ is trivial. Assume now that $i\geq 1$. Using definitions, we can
write
$\br{b_i,b_j}=a\br{b_{i-1}b_0^{-r_{i-1}(\xi)},\,b_{j-1}b_0^{-r_{j-1}(\xi)}}a^{-1}$.
By induction hypothesis the relation $\br{b_{i-1},b_{j-1}}=1$ is a
consequence of relations $\br{b_0,b_k}=1$ with $1 \le k \leq j-1$.
Hence $\br{b_i,b_j}=1$ is a consequence of relations
$\br{b_0,b_k}=1$ with $1 \le k\leq j$. The induction is then complete.
Thus all relations $\br{b_i,b_j}=1$ with $1 \le i<j$ can be deleted
in presentation (\ref{EqPres}).

\item[Proof of $(2)$.] It remains to show that
$\ovBS=\Pres{a,b}{\br{b_0,b_i}=1,\, i \ge 1}$ is a minimal
presentation when $\xi=\pi(\xi)=0$. In this case we have
$r_i(\xi)=0$, and hence $b_i=a^ib^ma^{-i}$, for all $i \ge 1$. We
fix some $k\ge 1$ and we show that $\br{b_0,b_k}=1$ does not hold in
the group $G_k:=\Pres{a,b}{\br{b_0,b_i}=1,k\neq i \ge 1}$.

To prove this we consider the group
$\widetilde{B}=\Pres{g_0,\,g_1,\,\dots}{\br{g_i,g_j}=1,\, \vert i-j\vert
\neq k}$ and its subgroups $\widetilde{B}_m=\langle
g_0^m,\,g_1,\,g_2,\,\dots\rangle$ and $\widetilde{B}_{\xi}=\langle
g_1,\,g_2,\,\dots\rangle$. We will use basic facts on
\emph{graph groups}, i.e. groups defined by a presentation whose
relators are commutators of some pairs of the generators. (Such
presentations are often encoded by a graph, whose vertices
correspond to the generators and whose edges tell which ones
commute; this explains the terminology.) Let $G=\Pres{X}{R}$ be a
graph group. Any element of $G$ can be written as a word $c_1c_2
\dots c_l$ where each \it syllable \rm $c_i$ belongs to some cyclic
group generated by some element of $X$. We consider three types of
moves that we can perform on such words.
\begin{itemize}
\item[$1.$] Remove a syllable $c_i=1$.
\item[$2.$] Replace consecutive syllables $c_i$ and  $c_{i+1}$ in
the same cyclic subgroup by the single syllable $(c_i c_{i+1})$.
\item[$3.$] For consecutive syllables $c_i \in \langle x\rangle$ and $c_{i+1} \in \langle x' \rangle$ with
$x,x' \in X, \,\br{x,x'} \in R$, exchange $c_i$ and $c_{i+1}$.
\end{itemize}

If $g \in G$ is represented by a word $w=c_1 \dots c_l$ which cannot
be changed to a shorter word using any sequence of the above moves,
then $w$ is said to be a \it normal form \rm for $g$. We will use
the following results:

\begin{theorem} \rm \cite{Bau81} \it \label{GraGrou}
\item[\rm(Normal Form Theorem)]
 A normal form in a graph group represents the trivial element if
and only if it is the trivial word.
\item[\rm(Abelian Subgroups)]
 Any abelian subgroup of a graph group is a free abelian group.
\end{theorem}

We define the partial map $\psi:\widetilde{B}_{\xi} \longrightarrow
\widetilde{B}_{m}$ by $\psi(g_1)=g_0^m$ and $\psi(g_i)=g_{i-1}$ for
every $i\ge 2$. By using the Normal Form Theorem, we can readily show that
$\widetilde{B}_{\xi}=\Pres{g_1,\,g_2,\,\dots}{\br{g_i,g_j}=1, \vert
i-j\vert \neq k}$. We
clearly have $\br{\psi(g_i),\psi(g_j)}=1$ in $\widetilde{B}_m$ for
any $i,j\ge 1$ such that $\vert i-j \vert \neq k$. Hence $\psi$ induces
a surjective homomorphism from $\widetilde{B}_{\xi}$ onto
$\widetilde{B}_{m}$.

Let $g \in \widetilde{B}_{\xi}$. Replacing every $g_i$ by
$\psi(g_i)$ in any non-trivial normal form for $g$ in
$\widetilde{B}$ clearly leads to a non-trivial normal form for
$\psi(g)$ in $\widetilde{B}$. Hence $\psi$ is injective, which proves that
$\psi$ is an isomorphism. Let $\widetilde{\phi}$ be its inverse homomorphism. We set
$\widetilde{G}_k:=\Pres{\widetilde{B},a}{a
g a^{-1}=\widetilde{\phi}(g), g \in
\widetilde{B}_m}$. We trivially check that the map $a \mapsto
a,\,b \mapsto g_0$ induces a surjective homomorphism
from $G_k$ onto $\widetilde{G}_k$ that maps $\br{b_0,b_k}$ to
$\br{g_0,g_k}$. Since the graph group $\widetilde{B}$ embeds
isomorphically into the HNN extension $\widetilde{G}_k$, we have $\br{g_0,g_k} \neq 1$ and hence $\br{b_0,b_k} \neq
1$.
\end{sproof}

\section{C*-simplicity}
We first recall some definitions.
Let $\Gamma=\ovBS$ and let $\sigma: \ovBS \longrightarrow \Z$ be the
surjective homomorphism defined by $\sigma(a)=1$ and $\sigma(b)=0$. Let $\Z \wr
\Z=\Z \br{X^{\pm 1}} \rtimes_X \Z$ where the action of $\Z$ on the
additive group of $\Z \br{X^{\pm 1}}$ is the multiplication by $X$.
This group is generated by $\{(1,0),(0,1)\}$ and the map
$a \mapsto (0,1),\, b \mapsto (1,0)$ induces a surjective homomorphism $q_{m,\xi}$ from
$\Gamma$ onto $\Z \wr \Z$ \cite[Th. 3.12]{GS08}. By $(v)$ and $(vi)$ of Proposition \ref{PropBandBm}, 
the restriction of $q_{m,\xi}$ to $B$ (which identifies with a subgroup of $\Z \br{X}$) is the identity and the map
$\gamma \mapsto P_{\gamma}(\xi/m)$ is a well-defined homomorphism from $B$ to the additive group of $\Z_{m}$.

Given a group $\Gamma$, recall that its \emph{reduced C*-algebra
$C_r^{\ast}(\Gamma)$} is the closure for the operator norm of the
group algebra $\C \br{\Gamma}$ acting by the left-regular
representation on the Hilbert space $\ell^2(\Gamma)$. For an
introduction to group C*-algebras, see for example
\cite[Ch.VII]{Dav96}. A group is C*-simple if it is infinite and if
its reduced C*-algebra is a simple topological algebra.

 Non-abelian free groups are C*-simple. The first proof of this
 fact, due to Powers \cite{Pow75}, relies on a combinatorial
 property of free groups shared by many other groups, called for
 this reason \emph{Powers groups}. Thus Powers groups are
 C*-simple. The Baumslag-Solitar group $BS(m,n)$ is C*-simple if and only if $|m| \neq |n|$ \cite[Th. 4.9]{Ivan07}. 
In this case, it is actually a strongly Powers group \cite[Pr. 5]{HP09}. 
A group $G$ is said to be \emph{strongly Powers} if any of its subnormal subgroups is a Powers group.

 \begin{theorem} \label{ThPowers}
 Let $|m|>1,\,\xi \in \Z_m$. Then $\ovBS$ is a strongly
 Powers group.
 \end{theorem}

Amenable groups are not C*-simple \cite[Pr. 3]{Har07}. This is the reason why we have to exclude the groups $\ovBS$ whith $|m|=1$. Actually, there is only one such marked group \cite[Th. 2.1]{GS08} and it is isomorphic
to the solvable group $\Z \wr \Z$ \cite[Th. 2]{Sta06a}.

 We consider the action of $\Gamma=\ovBS$ on its Bass-Serre tree $T=X_{m,\xi}$ and use the criterion of de
la Harpe and Pr\'eaux \cite[Pr. 16]{HP09} on tree action to show
Theorem \ref{ThPowers}. The latter criterion needs the action of
$\Gamma$ to be faithful in a strong sense: it has to be
\emph{slender}. Two other conditions are required, but both follow
immediatly from \cite[Pr. 24]{HP09}, as $\Gamma$ is a non-ascending
HNN extension. Let us define a slender action. A tree automorphism $\gamma$
of $T$ is \emph{slender} if its fixed point set $(\partial
T)^{\gamma}$ has empty interior in $\partial T$ with respect to the shadow topology (Sec. \ref{SubSecHNN}). The action of
$\Gamma$ on $T$ is slender if for every $\gamma \in \Gamma \setminus
\{1\}$ the automorphism of $T$ induced by $\gamma$, also denoted by
$\gamma$, is slender. A slender action is faithful, it is even
strongly faithful in the sense of \cite[Sec. 1]{HP09}.

 Since hyperbolic elements are obviously slender, we focus now on the fixed point set of elliptic
elements. We still need more definitions to describe the fixed point set $T^{\gamma}$
for any $\gamma \in B$.

Let $l,u \in \Z \cup \{\pm \infty\}$. We denote by $\{l \le \sigma
\le u\}$ the subgraph of $T$ whose vertices $\gamma B$ satisfy $l
\le \sigma(\gamma) \le u$. We denote by $T\br{l,u}$ the connected
component of $1B$ in $\{l \le \sigma \le u\}$.

\begin{lemma} \label{LemPartialSigm}
Let $l,u \in \Z \cup \{\pm \infty\}$. Assume either that $l>-\infty$ and
$|m|>1$, or that $u<+\infty$. Then, the set $\partial \{l \le \sigma \le u\}
\subset \partial T$ has empty interior in $\partial T$ with respect to the
shadow topology.
\end{lemma}
\begin{proof}
 Assume that $l>-\infty$ and $|m|>1$. Since any vertex $\gamma B$ of $T$ has
$|m|$ neighbours $\gamma' B$ such that $\sigma(\gamma')=\sigma(\gamma)-1$, any
shadow contains (the class of) a geodesic ray $(\gamma_1 B,\gamma_2 B,\ldots,
\gamma_n B,\ldots)$ such that $\sigma(\gamma_n)$ tends to $-\infty$. Such a ray
does not lie in $\partial \{l \le \sigma \le u\}$. The proof of the second case
is analogous (any vertex $\gamma B$ of $T$ has countably many
neighbours $\gamma'B$ such that $\sigma(\gamma')=\sigma(\gamma)+1$).
\end{proof}

Let $\nu$ be the natural valuation on $\Z\br{X}$, i.e. the
 one defined by $\nu(X^i)=i$ for every $i$. Let $\nu$ be the
 ``valuation" defined on $B$ by $\mu(\gamma)=\sup \{i \ge 1:
 P_{\gamma}(\xi/m) \in d{\hat{m}}^i \Z_{m}\}$ where $\sup \emptyset=0$.
Let $\gamma \in B$. Observe that
\begin{itemize}
\item $\gamma \in B_m$ if and only if $\mu(\gamma) \ge 1$,
\item $\gamma\in B_{\xi}$ if and only if $\nu(P_{\gamma}) \ge 1$.
\end{itemize}

\begin{proposition} \label{PropFixEll}
Let $\gamma \in B$. We have
$T^{\gamma}=T\br{-\mu(\gamma),\nu(P_{\gamma})}$.
\end{proposition}

\begin{corollary} \label{CorSlender}
Let $\Gamma=\ovBS$.
\begin{itemize}
\item The action of $\Gamma$ on $T$ is slender.
\item The centralizer of $b$ in $\Gamma$ coincides with $B$ if $|m|>1$.
\item The centralizer of $a^k$ in $\Gamma$ coincides with $\langle a
\rangle$ for every $k \neq 0$.
\end{itemize}
\end{corollary}

\begin{sproof}{Proof of Corollary \ref{CorSlender}}
Let us show that the action of $\Gamma$ is slender. Remind that any
hyperbolic element of $\Gamma$ is slender. Since any elliptic
element is conjugated to some $\gamma \in B$ by a homeomorphism of
$\partial T$, it suffices to prove that every $\gamma \in B\setminus
\{1\}$ is slender. Since $\nu(P_{\gamma})<\infty$ for every $\gamma
\in B\setminus \{1\}$, the result follows from Proposition
\ref{PropFixEll} and Lemma \ref{LemPartialSigm}.

Assume that $|m|>1$ and let $\gamma$ be an element of the
centralizer of $b$ in $\Gamma$. By Proposition \ref{PropFixEll}, we
have $T^b={1B}$. Since $\gamma$ commutes with $b$, we have then
$\gamma \cdot 1B=1B$, i.e. $\gamma \in B$.

Let $k \in \Z\setminus \{0\}$ and let $\gamma$ be an element of the centralizer of
$a^k$ in $\Gamma$. Let $a_{+}$ (resp. $a_{-}$) be the class of rays
which are cofinal with $(a^nB)_{n \ge 0}$ (resp. $(a^nB)_{n \le
0}$). Then $\gamma$ preserves the fixed point set $\{a_{-},a_{+}\}$
of $a^k$. Since $\sigma(\gamma a^n)=\sigma(\gamma)+n$ for every $n \in \Z$,
$\gamma$ cannot exchange $a_{+}$ and $a_{-}$ and hence fixes them both.
Consequently, there is some $n \in \Z$ such that $\gamma \cdot B=a^{n}B$, i.e.
there is $g \in B$ such that $\gamma=a^{n}g$. We deduce that $g$ centralizes
$a^k$ and hence $g$ fixes $a_+$ and $a_-$. As $g$ is elliptic, Proposition
\ref{PropFixEll} gives $\nu(P_g)=+\infty$, hence $g=1$. Thus $\gamma=a^n$.
\end{sproof}

Proposition \ref{PropFixEll} is a straightforward consequence of the
following lemma.
\begin{lemma} \label{LemFixEll}
Let $g\in\Gamma$ and let $c_0a^{\epsilon_1} c_1 \cdots
a^{\epsilon_k}c_k$ with $\epsilon_i \in \{\pm 1\},\, c_i \in B$ be a
reduced form of $g$. Let $\sigma_i(g)=\sum_{j=1}^i \epsilon_j$. Let
$\gamma \in B$. The following are
 equivalent:
\begin{itemize}
\item[$(i)$] $gB \in T^{\gamma}$,
\item[$(ii)$] $g^{-1}\gamma g \in B$,
\item[$(iii)$] $ -\mu(\gamma) \le \sigma_i(g) \le \nu(P_{\gamma})$ for
every $1 \le i \le k$.
\end{itemize}
If the previous conditions hold then we have: $g^{-1}\gamma
g=a^{-\sigma(g)}\gamma a^{\sigma(g)}$.
\end{lemma}

\begin{proof}
$(i)\Leftrightarrow(ii)$ is trivial. We show the equivalence $(ii)
\Leftrightarrow (iii)$ and the last statement of the lemma by
induction on $k$. If $k=0$, both are trivial. Assume that $k>0$ and
write $w=c_0a^{\epsilon_1}g'$.

$(ii) \Rightarrow (iii)$: we have ${g'}^{-1}\gamma' {g'} \in B$ with
$\gamma'=a^{-\epsilon_1}\gamma a^{\epsilon_1}$. By Britton's lemma,
we have: ${g'}^{-1}\gamma' {g'} \in B$ and, either $\gamma \in B_m$
and $\epsilon_1=-1$, or $\gamma \in B_{\xi}$ and $\epsilon_1=1$. We
deduce that $\gamma' \in B$,
$\mu(\gamma')=\mu(\gamma)+\epsilon_1,\,\nu(P_{\gamma'})=\nu(P_{\gamma})-\epsilon_1,\,\gamma'
\in B$ and $-\mu(\gamma) \le \epsilon_1 \le \nu(P_{\gamma})$. The
result then follows from the induction
hypothesis.\\
$(iii) \Rightarrow (ii)$: As $-\mu(\gamma) \le \epsilon_1 \le
\nu(P_{\gamma})$, we have either $\epsilon_1=-1$ and hence $\gamma
\in B_m$, or $\epsilon_1=1$ and hence $\gamma \in B_{\xi}$. We
deduce that $\gamma' \in B$. The result follows from the induction
hypothesis.

Last statement: by induction hypothesis, we have $${g'}^{-1}\gamma'
{g'}=a^{-\sigma(g')}\gamma'
a^{\sigma(g')}=a^{-\epsilon_1-\sigma(g')}\gamma
a^{\epsilon_1+\sigma(g')}=a^{-\sigma(g)}\gamma a^{\sigma(g)}.$$
\end{proof}

\begin{sproof}{Proof of Theorem \ref{ThPowers}}
Since $\Gamma=\ovBS$ is a non-ascending HNN extension, the action of
$\Gamma$ is strongly hyperbolic on $T$ and minimal on $\partial T$
\cite[Pr. 22]{HP09}. The action of $\Gamma$ on $T$ is slender by
Corollary \ref{CorSlender}. By \cite[Pr. 16]{HP09}, $\Gamma$ is a strongly
Powers group.
\end{sproof}

\paragraph{Inner amenability} \label{SecInner}

A countable group $G$ is said \emph{inner amenable} if it admits a
\emph{mean} (i.e. a non-zero, finite and finitely additive measure)
on the set of all the subsets of $G\setminus\{1\}$ which is
invariant under inner automorphisms. We say that $G$ has the \emph{icc} property if the conjugacy class of any of its non-trivial element is infinite. These two notions are motivated by the study of the von Neumann algebra $W^{\ast}(G)$ of $G$ (see \cite{Eff75,BH86}).  Amenable groups or groups that do not have icc are clearly inner amenable. The second-named author has proved that the Baumslag-Solitar group $BS(m,n)$ has icc, is inner amenable but not amenable whenever $|m|>|n|>1$ \cite[Ex.2.4 and 3.2]{Sta06b}. Note that for every $|m|>1$, the group $\BS$ also has icc since every Powers group does \cite[Pr. 1]{Har85} 

\begin{proposition} \label{PropInnAmen}
Let $|m|>1,\,\xi \in \Z_m$. The group $\BS$ is inner amenable and non-amenable.
\end{proposition}

The proof relies on:

\begin{theorem} \rm \cite[Pr. A.0.2]{Sta06b} \it \label{ThSta}
Let $\Gamma=HNN(\Lambda,H,K,\phi)$ with $H \neq \Lambda$ or $K \neq
\Lambda$. Let $Z(\Lambda)$ be the center of $\Lambda$.
 If for every $n \ge 1$ there exist some non-trivial elements
$h_0^{(n)},h_1^{(n)},\dots,h_n^{(n)} \in Z(\Lambda) \cap H \cap K$ such
that $h_i^{(n)}=\phi(h_{i-1}^{(n)})$ for $i=1,\dots,n$, then
$\Gamma$ is inner amenable.
\end{theorem}

\begin{sproof}{Proof of Proposition \ref{PropInnAmen}}
Let $n \ge 1$ and set $h_i^{(n)}:=a^ib^{m^{n+1}}a^{-i}$ for $i=0,\dots,n.$ The hypotheses of Theorem \ref{ThSta}.$ii$ are
trivially satisfied by these elements, which proves that $\ovBS$ is inner amenable. Using Britton's Lemma, we can readily show that the subgroup generated by $a$ and $bab^{-1}$ is a non-abelian free subgroup of $\ovBS$. Hence $\ovBS$ is not amenable by a classical result of von Neumann.
\end{sproof}

Actually, we can prove a stronger statement than the existence of the elements $h_i^{(n)}$ as in Theorem \ref{ThSta}. Indeed, it follows easily from Lemma \ref{LemFixEll} that $\bigcap_{1 \le i \le n}\gamma_i B \gamma_i^{-1}$ is a free abelian group of infinite countable rank for every $n$ and every $\gamma_1,\,\dots,\,\gamma_n \in \ovBS$. By \cite[Lem. 1.1]{Kro90}, this yields the following vanishing cohomological property: for every free $\Z \br{\ovBS}$-module $F$, we have $H^{i}(\ovBS,F)=0$ for all $i$. 
\section{Homomorphisms} \label{SecHomo}
This section is devoted to the study of group homomorphisms from a given limit group to another one. We classify the limits of Baumslag-Solitar groups $\ovBS$  up to abstract group isomorphism (Theorem \ref{ThClass}), we compute the automorphism group of every limit (Proposition \ref{PropAut}) and we prove that every limit is hopfian (Theorem \ref{ThHopf}). Finally, we show that every limit has infinite twisted conjugacy classes (Proposition \ref{CorTwist})  

As the map $a \mapsto a,\, b \mapsto b$ induces an isomorphism from $BS(m,n)$ to $BS(-m,-n)$, it also induces an isomorphism from
$\ovBS$ to $\bBS{-m}{-\xi}$ for any $\xi \in \Z_m$. For these reason, we will assume that $m>0$.

The following lemma can be readily deduced from the definition of $\ovBS$. 
\begin{lemma} \label{LemEmbedEta}
Let $d=\gcd(m,\xi)$ and let $\eta=\pi(\xi/d)$ where $\pi:\Z_m \longrightarrow \Z_{\mh}$ is the canonical map. 
The map $a \mapsto a,\,b \mapsto b^d$ induces an injective homomorphism from $\bBS{\mh}{\eta}$ into $\ovBS$.
\end{lemma}

From now on, we will consider $\widetilde{BS}(m,\xi)$ rather than $\ovBS$ because it will ease off notations. We fix 
$m,m' \in \N \setminus \{0\}$ and $\xi \in \Z_m,\,\xi' \in \Z_{m'}$ and we set
$\Gamma=\tBS{m}{\xi}$ and $\Gamma'=\tBS{m'}{\xi'}$. 

The following proposition shows that every surjective homomorphism from $\Gamma$ onto $\Gamma'$ is conjugated to a simple one by an element of $\Gamma'$.
\begin{proposition} \label{PropEpi}
Let $p:\Gamma \longrightarrow \Gamma'$ be a surjective
homomorphism. Then $m'$ divides $m$, $\sigma(p(a))=\pm 1$ and $p(e_0)$ is conjugated to $\pm e_0$. Moreover $\sigma(p(a))=1$ if $m'>1$. 
\end{proposition}
Before proving Proposition \ref{PropEpi}, we quote the following observation
for further reference.
\begin{remark}\label{AutomSpecZZ}
 For any $Q\in \Z \br{X^{\pm 1}}$, there exists $\psi_Q\in \operatorname{Aut}(\Z
\br{X^{\pm 1}}\rtimes \Z)$ such that $\psi_Q(1,0)=(1,0)$ and
$\psi_Q(0,1)=(Q,1)$. This automorphism satisfies $\psi_Q(P,0)=(P,0)$ for all
$P\in \Z \br{X^{\pm 1}}$.
\end{remark}
Indeed, this can be readily checked with the well-known
presentation
\[
 \Z \br{X^{\pm 1}}\rtimes \Z = \Z\wr\Z = \Pres{a,b}{[a^iba^{-i},b]=1 \text{ for all }
i \ge 1} \ ,
\]
where $a$ corresponds to $(0,1)$ and $b$ corresponds to $(1,0)$.

\begin{proof}[Proof of Proposition \ref{PropEpi}]
We set $\alpha=p(a),\, \beta=p(e_0)$ and $\gamma=\alpha \beta
\alpha^{-1}$. First we show that $\beta$ is an elliptic element of
$\Gamma'$. Assume by contradiction that $\beta$ is hyperbolic. Then
$\gamma^m$ is hyperbolic with axis $\alpha(D)$, where $D$ is the
axis of $\beta$. Since $e_1=a(me_0)a^{-1}$ commutes with $e_0$ in
$\Gamma$, $\beta$ commutes with $\gamma^m$. As a result, $\gamma^m$
has the same axis as $\beta$, namely $D$. Thus $D$ is invariant
under $\alpha$. There are two cases: either $\alpha$ is hyperbolic
(case 1) or $\alpha$ is elliptic (case 2).

Case 1: the bi-infinite ray $D$ is then the axis of $\alpha$. As
$\alpha$ and $\beta$ generate $\Gamma'$, the two ends of $D$ are
fixed ends of $\Gamma'$. This impossible since a non-degenerate 
HNN extension has at most one fixed end on the boundary of its Bass-Serre tree.

Case 2: the automorphism $\alpha'=\alpha \beta$ is hyperbolic, since $\Gamma'$
cannot be generated by two elliptic elements. Indeed, the set of
elliptic elements is contained in the kernel of
$\sigma:\Gamma' \rightarrow \Z$. So, the argument used in case 1 applies to $\alpha'$ and $\beta$.

Therefore, up to conjugacy, we can assume that $\beta \in E$. Now we use the
surjective homomorphism $q_{m',\xi'}:\Gamma' \longrightarrow \Z \wr
\Z$ (see Proposition \ref{PropBandBm}.$v$). Since $\sigma(\beta)=0$ and $\{\sigma(\alpha),\sigma(\beta)\}$
generates $\Z$, we have $\sigma(\alpha)=\epsilon'$ with $\epsilon'
=\pm 1$. We can write $q(\beta)=(P,0),\, q(\alpha)=(Q,\epsilon')$
with $P,Q \in \Z\br{X^{\pm 1}}$. If $\epsilon'=1$, the automorphism
$\psi_{-Q}$, defined as in Remark \ref{AutomSpecZZ}, maps $(P,0)$ to
$(P,0)$ and $(Q,1)$ to $(0,1)$; if $\epsilon'=1$, the automorphism
$\psi_{XQ}$ maps $(P,0)$ to $(P,0)$ and  $(Q,-1)$ to $(0,-1)$.
In both cases, the image $\{(P,0),\,(0,\epsilon')\}$ of
$\{q(\alpha),\,q(\beta)\}$ by the latter automorphism generates $\Z \wr \Z$,
i.e. the subgroup $P \Z \br{X^{\pm 1}} \rtimes \Z$ coincides with $\Z \wr \Z$.
This implies that $P$ is invertible in $\Z \br{X^{\pm 1}}$, which gives
$P=\epsilon X^i$ for some $\epsilon=\pm 1,i \in \Z$. As $P=q(\beta)$ lies in
$\Z[X]$, we have $i\in \N$, whence $q(\beta)=q(a^i(\epsilon e_0)a^{-i})$. Since
$q$ is injective on $E$ by Proposition \ref{PropBandBm}, we deduce that
$\beta=a^i (\epsilon e_0)a^{-i}$. Thus, up to conjugacy, we can assume that
$\beta=\epsilon e_0$.

 Assume now that $m'>1$. Let $\alpha=c_1 a^{\epsilon_1} c_2 a^{\epsilon_2} \cdots c_l
a^{\epsilon_l}c_{l+1}=za^{\epsilon_l}c_{l+1}$
 (with $\epsilon_i=\pm1$ and $c_i \in E$) be a reduced form in $\Gamma'$.
As $\gamma^m=\alpha (m\epsilon e_0) \alpha^{-1}$ commutes with $\beta=\epsilon
e_0$, we
deduce  from Corollary \ref{CorSlender} that $\alpha (m e_0) \alpha^{-1}
\in E$.
It follows that $\epsilon_l=1$ and $m'$ divides $m$ by Britton's lemma. As
$a^{\sigma(\alpha)}(me_0)a^{-\sigma(\alpha)}\in E$ by Lemma \ref{LemFixEll}, it
follows from Britton's lemma that $\sigma(\alpha) \ge 0$. Since
$\sigma(\alpha)\in \{\pm 1\}$, we deduce that $\sigma(\alpha)=1$.
\end{proof}

Let $G$ be a group and let $\phi$ be an automorphism of $G$. Two elements
$\gamma,\,\gamma' \in  G$ are said to be \emph{$\phi$-twisted conjugate} if
there is $g \in G$ such that $\gamma'=g \gamma \phi(g^{-1})$. We say that $G$
has infinitely many twisted conjugacy classes if $G$ has infinitely many
$\phi$-twisted conjugacy classes for every automorphism $\phi$. The study of
this property is mainly motivated by topological fixed point theory and by the
problem of finding a twisted analogue of the classical Burnside-Froebenius
theorem (see \cite{FH94} for a introduction to these topics). Baumslag-Solitar
groups $BS(m,n)$ with $(m,n) \neq \pm(1,1)$ \cite{FG06,FG08} and the wreath
product $\Z \wr \Z$ \cite[Cor. 4.3]{GW06} have infinitely many twisted conjugacy
classes (the reader may consult \cite{Rom09} for an up-to-date list of known
examples).

\begin{corollary} \label{CorTwist}
 The group $\Gamma$ has infinitely many twisted conjugacy classes.
\end{corollary}

\begin{proof}
If $m=1$, one has $\Gamma = \Z\wr\Z$. As mentioned above, it has infinitely many
twisted conjugacy classes.

Suppose now that $m>1$ and let $\phi$ be an automorphism of $\Gamma$. It
follows from Proposition \ref{PropEpi} that $\sigma \circ \phi=\sigma$. Thus
$\sigma$ is constant on each $\phi$-twisted conjugacy class. As $\sigma$ takes
infinitely many values, $\Gamma$ has infinitely many $\phi$-twisted conjugacy
classes.
\end{proof}

For every $i \ge 1$, we set $w_i=a^{i} (me_0)  a^{-1}  (-r_1(\xi)e_0 ) a^{-1} 
(-r_2(\xi)e_0 ) a^{-1} (-r_{i-1}(\xi) e_0)  a^{-1}$.
Then $e_i=w_i$ holds in $\Gamma$ for every $i \ge 1$. Recall that
\begin{eqnarray} \label{EqPresE1}
\Gamma=&\Pres{a,e_0,e_1,\dots}{\begin{array}{l}
[e_i,e_j]=1 \text{ for all } i,j \ge 0, \\
a(me_0)a^{-1} =e_1,\,a(e_i-r_i(\xi)e_0)a^{-1} = e_{i+1} \text{ for all } i \ge
1 \end{array}}\\ \label{EqPresE2}=&\Pres{a,e_0}{\br{e_0,w_i}=1 \text{ for all }
i \ge 1}
\end{eqnarray}
by Corollary \ref{CorHNN2} and Theorem \ref{ThPres}.

Let $J$ be the map defined by $J(a)=a$ and $J(e_0)=-e_0$. It follows from presentation (\ref{EqPresE1}) that $J$ induces an automorphism of $\Gamma$ such that $J(e_i)=-e_i$ for every $i$.
\begin{lemma} \label{LemWinE}
Let $t_1,\dots,\,t_n \in \{0,\dots, m-1\}$ and let 
$$w(m,t_1,\dots,t_n)=a^{n+1}  (me_0)  a^{-1}  (-t_1e_0 ) a^{-1}  (-t_2e_0 ) a^{-1} \cdots (-t_n e_0)  a^{-1}.$$
Then the following are equivalent:
\begin{itemize}
\item[$(i)$] The image of $w(m,t_1,\dots,t_n)$ in $\Gamma$ belongs to $E$,
\item[$(ii)$] The image of $w(m,t_1,\dots,t_n) e_0  w(-m,-t_1,\dots,-t_n)(-e_0)$ in $\Gamma$ is trivial,
\item[$(iii)$] $t_i=r_i(\xi)$ for every $1\le i \le n$.
\end{itemize}
\end{lemma}

\begin{proof}
$(i)\Rightarrow(ii)$: As $J(w(m,t_1,\dots,t_n))=w(-m,-t_1,\dots,-t_n)$ and $J$ maps any element of $E$ to its inverse, we deduce that
$w(m,t_1,\dots,t_n)e_0w(-m,-t_1,\dots,-t_n)(-e_0)=1$ in $\Gamma$.

$(ii)\Rightarrow(iii)$:
Assume that $w=w(m,t_1,\dots,t_n) e_0 w(-m,-t_1,\dots,-t_n)(-e_0)$ has a trivial image in $\Gamma$. We prove the following claim by induction: for every $1 \le i \le n$, we have $w=v_ie_0J(v_i)(-e_0)$ with $v_i=a^{n+1-i}  (e_i-t_ie_0)  a^{-1}  (-t_{i+1}e_0 ) a^{-1} \cdots (-t_ne_0)  a^{-1}$ and $t_j=r_j(\xi)$ for every $1 \le j \le i$.
For $i=1$, the claim follows from the relation $a  (me_0)  a^{-1}=e_1$ in $\Gamma$.
Assume that the claim holds for some $i \ge 1$. We have then $w=v_ie_0J(v_i)(-e_0)=1$ in $\Gamma$. By Britton's lemma, the sequence $(a,\,e_i-t_ie_0,\,a^{-1})$ is not reduced and hence $e_i-t_ie_0 \in E_{m,\xi}$. Therefore $t_i=r_i(\xi)$ and $a  (e_i-t_i e_0)  a^{-1}=e_{i+1}$. We deduce that $w=v_{i+1}e_0J(v_{i+1})$ with $v_{i+1}=a^{n-i}  (e_{i+1}-t_{i+1}e_0)  a^{-1}  (-t_{i+2}e_0 )\cdots (-t_ne_0)  a^{-1}$, which completes the induction.

$(iii)\Rightarrow(i)$: it follows from the fact that $w(m,r_1(\xi),\dots,r_i(\xi))=w_i$ and $w_i=e_i$ in $\Gamma$ for $i \ge 1$.
\end{proof}

\begin{lemma} \label{LemTheta}
Assume that $m=m'> 1$. Let $\theta$ be a map such that  $\sigma \circ \theta(a)=1$ and $\theta(e_0)= e_0$. If $\theta$ induces an homomorphism from $\Gamma$ to $\Gamma'$ then $r_i(\xi)=r_i(\xi')$ for every $i \ge 1$. In this case, the restriction of $\theta$ to $E$ is the identity and we have $|\theta(\gamma)|_{a}=|\gamma|_a|\theta(a)|_a$ for every $\gamma \in \Gamma$. In particular $\theta$ is injective.
\end{lemma}

\begin{proof}
We can write $\theta(a)=zae$ with $e \in E$ and $z \in \Gamma'$ such that $\sigma(z)=0$.
Assume that $\theta$ induces an homomorphism from $\Gamma$ to $\Gamma'$. We show by induction that $\theta(e_i)=w_i$ for every $i \ge 1$. First, observe that $\theta(e_i)$ commutes with  $\theta(e_0)=e_0$ for every $i \ge 1$. Hence $\theta(e_i) \in E$ for every $i \ge 1$ by Corollary \ref{CorSlender}. In particular, $\theta(e_1)=\theta(w_1)=z  a  (me_0) a^{-1}  z^{-1}\in E$. As $\sigma(z)=0$, we deduce from Lemma \ref{LemFixEll} that $\theta(e_1)=w_1$. Assume now that $\theta(e_i)=w_i$ for some $i \ge 1$. As $e_{i+1}=a ( e_i-r_i(\xi)e_0)  a^{-1}$ in $\Gamma$, we have $\theta(e_{i+1})=z a  (w_i -r_i(\xi)e_0)  a^{-1}  z^{-1} \in E$. We obtain $\theta(e_{i+1})=w_{i+1}$ by Lemma \ref{LemFixEll}, wich completes the induction. As $w_i=\theta(e_i) \in E$ for every $i \ge 1$, we deduce from Lemma \ref{LemWinE} that $r_i(\xi)=r_i(\xi')$ for every $i \ge 1$.

Assume that the latter conditions holds and let $\theta(a)=c_1  a^{\epsilon_1} 
c_2 a^{\epsilon_2} \cdots c_l  a^{\epsilon_l}c_{l+1}$ (with $\epsilon_j=\pm1$
and $c_j \in E$ for every $j$) be a reduced form in $\Gamma'$. Since $\theta(a) 
(me_0) \theta(a)^{-1}=e_1$ and $\theta(a)^{-1}  e_1 \theta(a)=me_0$, we have
$\epsilon_1=\epsilon_l=1$ by Britton's lemma. It readily follows that
$|\theta(\gamma)|_{a}=|\gamma|_a|\theta(a)|_a$ for every $\gamma \in \Gamma$. 
\end{proof}

\paragraph{Hopf property and residual finiteness} \label{ParHopf}

An important reason for considering Baumslag-Solitar groups was, at the origin, that they gave the first examples of non-hopfian one-relator groups \cite{BS62}. Let us recall that a group $G$ is \emph{hopfian} if every surjective endomorphism from $G$ is an isomorphism. It is known that the Hopf property is neither open \cite{ABL+05,Sta06a} nor closed \cite[Pr. 5.10]{CGP07} in the space of marked groups. 

A group $G$ is said to be \emph{residually finite} if for every non-trivial element $g$ of $G$ there is a finite quotient $F$ of $G$ such that the image of $g$ in $F$ is non-trivial. For $m=1$, the only limit is $\tBS{1}{0} = \Z\wr\Z$. This group is residually finite, hence hopfian by a well-known theorem of Malcev; see e.g. \cite[Theorem IV.4.10]{LS77}. 

\begin{definition}
 A group $G$ is \emph{co-hopfian} if every injective homomorphism from $G$ to itself is an isomorphism.
\end{definition}

\begin{theorem} \label{ThHopf}
The group $\Gamma$ is hopfian but not co-hopfian.
\end{theorem}
We denote by $J$ the automorphism of $\Gamma$ defined by $J(a)=a$ and $J(e_0)=-e_0$.

\begin{proof}
We can assume that $m>1$.
Let $p$ be a surjective endomorphism of $\Gamma$. By Proposition \ref{PropEpi}, there is $\epsilon \in \{0,1\}$ and an innner automorphism $\tau$ of $\Gamma$ such that $p'=J^{\epsilon} \circ \tau \circ p$ satisfies $\sigma \circ p'(a)=1$ and $p'(e_0)=e_0$ with. By Lemma \ref{LemTheta}, $p'$ is injective and hence so is $p$. Therefore $p$ is an isomorphism.

Let $k \in \N \setminus \{1\}$ be coprime  with $m$ and let $\theta$ be the map
defined by $\theta(a)=a,\,\theta(e_0)=ke_0$. Considering the group presentation
(\ref{EqPresE1}), we deduce that $\theta$ induces an endomorphism of $\Gamma$.
We can readily check that $\theta(e)=ke$ for every $e \in E$ and that
$|\theta(\gamma)|_{a}=|\gamma|_a$ for every $\gamma \in \Gamma$. It follows that
$\theta$ is injective. We also deduce that $\theta(\Gamma)\cap E = \theta(E) = kE
\neq E$. Hence $\theta$ is not surjective.
\end{proof}

We say that $\xi \in \Z_m=\bigoplus_{p|m} \Z_p$ is \emph{algebraic} if
the $p$-component $\xi_p \in \Z_p$ of $\xi$ is
algebraic over $\Q$ for every prime $p$ dividing $m$. It only means
that $\xi$ is a root in $\Z_m$ of some polynomial with coefficients in $\Z$. 
Let $d=\gcd(m,\xi)$ and $\eta=\pi(\xi/d)$ where $\pi:\Z_m \longrightarrow
\Z_{\mh}$ is the canonical map.  It readily follows from the definition of
Proposition \ref{PropBandBm}.$vi$ that $\xi$ (resp. $\eta$) is algebraic if and
only if $\ker \chi\neq 1$ (resp. $\ker \hat{\chi}\neq 1$). If $\xi$ is
invertible in $\Z_m$, i.e. $\gcd(m,\xi)=1$, then the two kernels coincide with
$\bigcap_{i \ge 0}a^{-i}E a^i$.

\begin{proposition} \label{PropResFin}
 If $\eta$ is algebraic then $\Gamma$ is not residually finite.
\end{proposition}

\begin{proof}
By Lemma \ref{LemEmbedEta}, $\overline{BS}(\mh,\eta)$ embeds isomorphically into $\ovBS$.
As a subgroup of a residually finite group is residually finite, we can assume
that $\gcd(m,\xi)=1$.
By Proposition \ref{PropBandBm}.$vi$, we can pick $e \in \bigcap_{i\ge 0}a^{-i}E
a^{i} \setminus E_1$. We set
$\gamma=\br{e,ae_0a^{-1}}=(-e)a(-e_0)a^{-1}eae_0a^{-1}$. By Britton's Lemma,
$\gamma$ is not trivial in $\Gamma$. We show that $\gamma$ has trivial image in
any finite quotient of $\Gamma$, which proves that $\Gamma$ is not residually
finite. Let $F$ be a finite quotient of $\Gamma$ with cardinal $n$. There is $e'
\in E$, such that $e=a^{-n+1}e'a^{n-1}$ in $\Gamma$. Since $a^n=1$ in $F$, we
have $\gamma=\br{a^{-n+1}e'a^{n-1},ae_0a^{-1}}=a\br{e',e_0}a^{-1}=1$ in $F$. 
The proof is then complete.
\end{proof}

\paragraph{Classification of limits up to group isomorphism}

\begin{theorem} \label{ThClass}
The group $\Gamma$ is isomorphic to $\Gamma'$ if and only if $m=m'$ and $r_i(\xi)=r_i(\xi')$ for every $i \ge 1$.
\end{theorem}

\begin{proof}
Assume that $\Gamma$ is isomorphic to $\Gamma'$. If $m'=1$, then $\Gamma'$ is
isomorphic to $\Z \wr \Z$ and so is $\Gamma$. This forces $m=1$ for $\Gamma$
would contain a non-abelian free subgroup otherwise. It follows that
$r_i(\xi)=r_i(\xi')=0$ for every $i \ge 1$. Therefore we can assume that $m'>1$.
By Proposition \ref{PropEpi}, $m'$ divides $m$ and there is an isomorphism
$\theta:\Gamma \longrightarrow \Gamma'$ such that $\sigma \circ \theta(a)=1$ and 
$\theta(e_0)=e_0$. Considering $\theta^{-1}$, we also deduce that $m$ divides $m'$ and hence
$m=m'$. By Lemma \ref{LemTheta}, we have $r_i(\xi)=r_i(\xi')$ for every $i \ge
1$.

The converse follows immediatly from the group presentation (\ref{EqPresE1}).
\end{proof}

\paragraph{Automorphism group}
Let $e \in E$ and let $\phi_e$ be the map $\Gamma$ defined by $\phi_e(a)=ae$ and
$\phi(e_0)=e_0$. We deduce from the group presentation (\ref{EqPresE2}) that
$\phi_e$ induces an automorphism of $\Gamma$ with inverse map $\phi_{-e}$.
Moreover, we have $J \circ \phi_e \circ J=\phi_{-e}$. The following lemma is
then immediate.
\begin{lemma} \label{LemDih}
 The automorphisms $\phi_{e_0}$ and $J$ generate a group isomorphic
 to an infinite dihedral group, namely the semi-direct product $\Z e_0 \rtimes \Z/2 \Z$ 
where the action of $\Z/2\Z$ on $E$ is multiplication by $-1$.
\end{lemma} 

Hence we can consider the semi-direct product $\Gamma \rtimes (\Z e_0 \rtimes \Z/2\Z)$ where the action 
of $\Z e_0  \rtimes \Z/2\Z$ on $\Gamma$ is the obvious one. We denote by
$\text{Inn}(\Gamma)$ the group of inner automorphisms of $\Gamma$ and by
$\text{Out}(\Gamma)=\text{Aut}(\Gamma)/\text{Inn}(\Gamma)$ the group of outer
automorphisms.

\begin{proposition} \label{PropAut}
Assume that $m>1$.
\begin{itemize}
\item $\text{Out}(\Gamma)$ is isomorphic to $\Z e_0 \rtimes \Z/2\Z$.
\item $\text{Aut}(\Gamma)$ is isomorphic to $\Gamma \rtimes (\Z e_0 \rtimes \Z/2\Z)$.
\end{itemize}
\end{proposition}

We denote by $C$ the subgroup of $E$ generated by $e_1-me_0$ and the elements $e_i-e_{i-1}+r_{i-1}e_0$ with $i \ge 1$.
Proposition \ref{PropAut} will follow from:
\begin{lemma} \label{LemPhiE}
Assume that $m>1$ and let $e \in E$.
\begin{itemize}
 \item The image of $\phi_e\circ J$ in $\text{Out}(\Gamma)$ is non-trivial.
 \item The image of $\phi_e$ is trivial in $\text{Out}(\Gamma)$ if and only if $e \in C$.
\item For every automorphism $\phi$ of $\Gamma$, there is $e \in E,\, \epsilon \in \{0,1\}$ such that $\phi=\phi_e \circ J^{\epsilon}$ holds in $\text{Out}(\Gamma)$.
\end{itemize}
\end{lemma}

\begin{proof} As $\Gamma$
is centerless (see e.g. Corollary \ref{CorSlender}) and torsion-free,
$\text{Inn}(\Gamma)$ is torsion-free. The first assertion follows from the fact that $\phi_e\circ J$ has order $2$.

Let $e \in E$ and assume that there is $z \in \Gamma$ such that
$\phi_e=\tau_z$. As $z$ centralizes $e_0$, we deduce from Corollary
\ref{CorSlender} that $z \in E$. We deduce from the equality
$\phi_{e}(a)=zaz^{-1}$ that $a(e+z)a^{-1}=z$. 
By Britton's lemma, we have $e+z \in E_{m,\xi}$ and hence $z \in E_1$.
Identifying $E$ with $B$, we deduce from Proposition \ref{PropBandBm}.$v$, that
$X(P_{e}(X)+P_{z}(X))=P_z(X)$. Therefore
$P_e(X)=-\frac{X-1}{X}P_z(X)=\iota(-P_z(X)) \in \mathfrak{C}$ where $\iota$ and
$\mathfrak{C}$ are defined in Proposition \ref{PropBandBm}.$vii$. Since
$\mathfrak{C}=q(C)$ and $q$ in injective, we have $e \in C$. Conversely, if $e
\in C$, we can readily check that $\phi_e=\tau_z$ where $z \in E_1$ is given by
the formula $P_z(X)=-\frac{X}{X-1}P_e(X)$.

Consider now an arbitrary automorphism $\phi$ and let us show that $\phi=\phi_e
\circ J^{\epsilon}$ holds in $\text{Out}(\Gamma)$ for some $e \in E$ and some
$\epsilon \in \{0,1\}$. By Proposition \ref{PropEpi}, we can assume that $\phi(e_0)=\pm e_0$ and $\phi(a)=zae'$ with $z \in \Gamma$ such that $\sigma(z)=0$ and $e' \in E$. Composing possibly by $J$, we can assume that $\phi(e_0)=e_0$
hence that $\phi$ and $\phi^{-1}$ both satisfy the conditions of Lemma
\ref{LemTheta}. We deduce from Lemma \ref{LemTheta} that
$1=|a|_a=|\phi(a)|_a|\phi^{-1}(a)|_a$. Therefore $|z|_a=0$, i.e. $z \in E$ and
hence $\phi=\phi_{e'+z}$ holds in $\text{Out}(\Gamma)$.
\end{proof}

\begin{sproof}{Proof of Proposition \ref{PropAut}}
 By Lemma \ref{LemPhiE}, $\text{Out}(\Gamma)$ is generated by the images of
$J,\,\phi_e$ with $e \in E$. By Proposition \ref{PropBandBm}.$vii$, the
quotient $E/C$ is infinite cyclic and generated by the image of $e_0$. Hence, by
Lemma \ref{LemPhiE}, the natural map
 $\text{Aut}(\Gamma)\rightarrow \text{Out}(\Gamma)$ induces an isomorphism from
the subgroup generated by $J$ and $\phi_{e_0}$ onto $\text{Out}(\Gamma)$. It
follows then from Lemma \ref{LemDih} that $\text{Out}(\Gamma)$ is isomorphic to
$\Z e_0 \rtimes \Z/2\Z$. As a result, the exact sequence $1 \rightarrow
\text{Inn}(\Gamma) \rightarrow \text{Aut}(\Gamma) \rightarrow \text{Out}(\Gamma)
\rightarrow 1$ splits. Since $\Gamma$ is centerless it naturally identifies with
$\text{Inn}(\Gamma)$. Thus $\text{Aut}(\Gamma)$ is isomorphic to $\Gamma \rtimes
(\Z e_0 \rtimes \Z/2\Z)$.
\end{sproof}

An immediate consequence of Proposition \ref{PropAut} is that every automorphism of $\Gamma$ is induced by an automorphism of $\Free(a,e_0)$. Group presentations with such property are called \emph{almost quasi-free presentations} \cite[Ch. II.2]{LS77}.

Recall that the map $a \mapsto (0,1),\, e_0 \mapsto (1,0)$ induces a surjective
homomorphism $q_{m,\xi}$ from
$\Gamma$ onto $\Z \wr \Z=\tBS{1}{0}$. Another consequence of Proposition \ref{PropAut} is:
\begin{corollary} \label{CorKerQ}
If $|m|>1$, then the kernel of $q_{m,\xi}$ is a characteristic free subgroup of
$\Gamma$ of infinite rank.
\end{corollary}

\begin{proof}
The normal subgroup $N=\ker q_{m,\xi}$ is a  free group by  \cite[Th.
3.11]{GS08}. As $N\nsubseteq E_{m,\xi}$ and $NE_{m,\xi}$ has infinite index in
$\Gamma$, $N$ is not finitely generated by \cite[Th. 9]{KS71}.

To conclude, thanks to Proposition \ref{PropAut}, it suffices to show that $N$
is invariant under the automorphisms $J$ and $\phi_{e_0}$. It is invariant
under $J$ since the diagram
\[
\xymatrix{
\Gamma \ar@{->>}[d]^{q_{m,\xi}} \ar[r]^{J} & \Gamma
\ar@{->>}[d]^{q_{m,\xi}} \\
\widetilde{BS}(1,0) \ar[r]^{J} & \widetilde{BS}(1,0) }
\]
commutes. A similar argument works for the automorphism $\phi_{e_0}$.
\end{proof}

\paragraph{Equationally noetherian groups} \label{SecEqNoe}

 In this section, we determine which groups $\BS$ are equationally noetherian. Equationally noetherian groups play an important role in \emph{algebraic geometry over groups} \cite{BMR99}, the state-of-the-art approach to equations over groups. An equationally neotherian group enjoys the following strong form of the Hopf property: any sequence of surjective endomorphisms is stationnary (see \cite[Th. D1.2]{MR00} or \cite[Cor. 2.8]{Hou07}).
Let us recall the definition. Given $w=w(x_1,\ldots,x_n) \in G\ast \mathbb F(x_1,\ldots,x_n)$ and a $n$-tuple $(g_1,\ldots g_n)\in G^n$, we denote by $w(g_1,\ldots,g_n)$ the element of $G$
 obtained by replacing $x_i$ by $g_i$. For any subset $W\subseteq G\ast \mathbb F(x_1,\ldots,x_n)$,
 we consider the \emph{roots}
\[
 \roots(W) = \{(g_1,\ldots,g_n)\in G^n : w(g_1,\ldots,g_n) = 1 \text{ for all } w\in W \} \ .
\]

\begin{definition}
 A group $G$ is \emph{equationally noetherian} if, for all $n\geqslant 1$ and 
for all $W\subseteq G\ast \mathbb F(x_1,\ldots,x_n)$, there exists a finite subset $W_0 \subseteq W$ such that $\roots(W) = \roots(W_0)$.
\end{definition}

Linear groups over a commutative, noetherian, unitary ring (e.g. a field), 
are equationally noetherian \cite{Bry77,Gub86} while any wreath product of 
a non-abelian group by an infinite one is not equationally noetherian \cite{BMR97}. 

\begin{proposition}\label{eqNbarBS}
 Let $m\in \Z\setminus \{0\}$ and $\xi\in \Z_m$. The group $\bBS m\xi$ is equationally noetherian if and only if $|m|=1$.
\end{proposition}

We need the following result on Baumslag-Solitar groups.
\begin{proposition} \cite[Pr. 5]{BMR99}\label{eqNBS}
  Let $m,n\in\Z\setminus \{0\}$.
 \begin{enumerate}
 \item If either $|m|=1$ or $|n|=1$ or $|m|=|n|$, then the group $BS(m,n)$ is linear over $\Q$ and hence equationally noetherian;
 \item else, the group $BS(m,n)$ is not equationally noetherian.
\end{enumerate}
\end{proposition}

As we need some excerpts of the proof of Proposition \ref{eqNBS}, we provide it in full. Let $R$ be a commutative ring with unity. We will use the following elementary fact without further mention. If a group $G$ has a finite index subgroup which is linear over $R$, then so is $G$ \cite[Lem. 2.3]{Weh73}.
\begin{proof}[Proof of Proposition \ref{eqNBS}(1)]
Suppose first that $|m|=1$ or $|n|=1$. It is well known, and easy to show, that the map
 $a \mapsto (x \mapsto \frac{m}{n}x),\, b \mapsto (x \mapsto x+1)$ yields an injective group homomorphism 
from $BS(m,n)$ into the affine group over $\Q$. The group $BS(m,n)$ is then linear over $\Q$. As $BS(m,n)$ is soluble in this case, we observe that it is linear over $\Z$ if and only if it is polycyclic\footnote{By theorems of Mal'cev and Auslander \cite[Ch. 2 and Ch. 3]{Seg83}, a soluble group is linear over $\Z$ if and only if it is polycyclic.}, i.e. $|n|=|m|=1$. 

 If $m=n$, it is easy to check that the normal subgroup $\langle\langle a, b^m \rangle\rangle$ 
is isomorphic to $\Free_{|m|}\times \Z$ and hence linear over $\Z$. Clearly, it has index $|m|$ in $BS(m,m)$. Thus $BS(m,n)$ is linear over $\Z$.

 Suppose finally $n=-m$. Let $BS_2(m,n)=\langle\langle a^2, b \rangle\rangle \subset BS(m,n)$. 
The subgroups $B_2(m,m)$ and $B_2(m,-m)$ are clearly isomorphic and have index two in $BS(m,m)$ and $BS(m,-m)$ respectively. 
We have shown that $BS(m,m)$ is linear over $\Z$. We deduce that $B_2(m,m)$ is linear over $\Z$ and hence so is $BS(m,-m)$.
\end{proof}
It follows from the above proof that $BS(m,n)$ is linear over $\Z$ if and only if $|m|=|n|$.

\begin{proof}[Proof of Proposition \ref{eqNBS}(2)]
 As the groups $BS(m,n)$ and $BS(n,m)$ are isomorphic, we may assume that $|m|<|n|$. 
Then, there exists $\nu>0$ and a prime number $p$ such that $p^\nu$ divides $n$ 
but not $m$. Let us consider the set
 \[
  W = \big\{w_i := [x^{-i} y x^{i}, z] : i\in\N \setminus \{0\} \big\} \subseteq \Free(x,y,z)
 \]
 and the triples $(x_k = a, y_k = b^{n^k}, z_k = b)$.\footnote{It is possible to use only one variable: 
replace $W$ by $W' = \{[a^{-i} y a^{i}, b] : i\in\N \setminus \{0\} \big\} \subseteq BS(m,n)\ast\Free(y)$.} If $n$ divides an integer $\beta$, then we have $a^{-1}b^\beta a = b^{\beta'}$ and the 
factorization of $\beta'$ contains (strictly) less factors $p$ than the one of $\beta$. 
Consequently, for all $k$, there exists $N(k)\in\N$ and $\alpha(k)\in\Z$ such that
 \[
    a^{-N(k)} b^{n^k} a^{N(k)} = b^{\alpha(k)} \text{ and } n \text{ does not divide } \alpha(k) \ .
 \]
 \begin{remark}\label{estim}
  \begin{enumerate}
     \item As $a^{-k}b^{n^k} a^k = b^{m^k}$, we have $N(k)\geqslant k$;
     \item Set $\mu=\mu(m)$ to be the maximal exponent arising in the factorization of $m$. Then we obtain 
     $N(k) \leqslant (\mu+2)k$. Indeed, 
we have $a^{-k}b^{n^k}a^k = b^{m^k}$ and the exponent of $p$ in the factorization of $m^k$ is at most $(\mu+1)k$. 
  \end{enumerate}
 \end{remark}
 The triple $(x_k = a, y_k = b^{n^k}, z_k = b)$ is a root of $w_i$ if and only if $i\leqslant N(k)$. Indeed:
 \begin{itemize}
    \item if $i\leqslant N(k)$, then $a^{-i} b^{n^k} a^{i}$ is a power of $b$, so that 
    $[a^{-i} b^{n^k} a^{i}, b] =  1$;
    \item if $i>N(k)$, then $[a^{-i} b^{n^k} a^{i}, b] = a^{-(i-N(k))} 
    b^{\alpha(k)} a^{i-N(k)} \cdot b \cdot a^{-(i-N(k))} b^{-\alpha(k)} 
    a^{i-N(k)} \cdot b^{-1}$. This is reduced in $BS(m,n)$, since $|m|>1$ and $n$ does not divide $\alpha(k)$.
 \end{itemize}
 If we now consider a finite subset $W_f = \{w_{i_1}, \ldots, w_{i_s}\} \subset W$, 
then, choosing $k$ large enough, we have $N(k) \geqslant i_1,\ldots, i_s$. Consequently, the triple $(x_{k}, y_{k}, z_{k})$ 
is in $\roots(W_f) \setminus \roots(W)$. This proves that $BS(m,n)$ is not
equationally noetherian.
\end{proof}
\begin{proof}[Proof of Proposition \ref{eqNbarBS}]
If $|m|=1$, one has $\bBS{m}{\xi}=\Z\wr\Z$ which is equationnally noetherian
(e.g. it is linear over the field $\Q(X)$).

Let us now assume that $|m|>1$. Consider a sequence $(\xi_n)$ of rational
integers 
such that $|\xi_n|\to \infty$ and $\xi_n \to \xi$ in $\Z_m$ for $n\to\infty$.
 We may assume that $|m|<|\xi_n|$ for all $n$. Set $W = \{w_i := [x^{-i} y x^{i}, z] : i\in\N \setminus \{0\}\}$ 
and $(x_k = a, y_k = b^{\xi_n^k}, z_k = b)$, as in the proof of Proposition \ref{eqNBS}(2). 
We have proved the existence of natural numbers $N_n(k)$ such that $(x_k, y_k, z_k)$ is a 
root of $w_i$ in $BS(m,\xi_n)$ if and only if $i\leqslant N_n(k)$. Moreover, Remark \ref{estim} gives the estimates
 \[
  k \leqslant N_n(k) \leqslant (\mu(m)+2)k
 \]
 for all $n$. Therefore, if we take a finite subset 
$W_f = \{w_{i_1}, \ldots, w_{i_s}\} \subset W$, then, choosing $k$ large enough, we have $N_n(k) 
\geqslant k \geqslant i_1,\ldots, i_s$ for all $n$. Therefore, for all $j$, we have 
$w_{i_j}(x_{k}, y_{k}, z_{k}) = 1$ in all groups $BS(m,\xi_n)$, and, passing to the limit, 
we see that $(x_k,y_k,z_k)$ is a root of $W_f$, in the group $\bBS m\xi$.
 
 On the other hand, by considering $w_i$ with $i> (\mu(m)+2)k$,
 we see that $w(x_{k}, y_{k}, z_{k})\neq 1$ in all $BS(m,\xi_n)$. Hence,
 in $\bBS{m}{\xi}$, the triple $(x_{k}, y_{k}, z_{k})$ is not a root of $W$.
 This proves that $\bBS{m}{\xi}$ is not equationally noetherian.
\end{proof}

\section{Dimensions}\label{SecDim}

In this section we give the first non-trivial Hausdorff dimension estimates of a subspace of the space of marked groups on two generators. Let us recall that the map $\overline{BS}_m:\Z_m \to \G \, ; \, \xi \mapsto \bBS{m}{\xi}$ is 
injective on $\Z_m^\times$ \cite[Th. 1]{GS08}. In order to estimate Hausdorff dimensions of the subspaces
\[
 Z_m^\times = \overline{BS}_m(\Z_m^\times) = \{ \bBS{m}{\xi} : \xi \text{ is invertible in } \Z_m \} \ ,
\]
we will prove that the maps between $Z_m^\times$ and $\Z_m^ \times$ satisfy H\"older conditions and then apply classical results about Hausdorff dimension. In this section, we always assume that $m$ is a rational integer satisfying $|m|\geqslant 2$.

\subsection{Distances between limits}

The first step towards Hausdorff dimension estimates is to estimate the distance between groups $\bBS m\xi$ and $\bBS{m}{\xi'}$ in terms in the $m$-adic distance between $\xi$ and $\xi'$.

\begin{theorem}\label{CongDistLim}
 Let $h \in \N \setminus \{0\}$ and $\xi,\xi'\in \Z_m$ satisfying $d:=\gcd(m,\xi) = \gcd(m,\xi')$. Setting $\hat m = m/d$, we have:
 \begin{enumerate}
  \item[(1)] If $\bBS m\xi$ and $\bBS m{\xi'}$ have the same relations up 
to length $2(|m|+1)h + 2|m| + 6$, then $\xi \equiv \xi' \ (\operatorname{mod}  \hat m^h d\Z_m)$;
  \item[(2)] If $\xi \equiv \xi' \ (\operatorname{mod}  \hat m^h d\Z_m)$, 
then $\bBS m\xi$ and $\bBS m{\xi'}$ have the same relations up to length $2h$.
 \end{enumerate}
\end{theorem}

\begin{sproof}{Proof of Theorem \ref{CongDistLim}}
Thanks to Corollary \ref{CorHNN2}, we may work in $\tBS m\xi$ and $\tBS m{\xi'}$ instead of $\bBS m\xi$ and $\bBS m{\xi'}$. 
Recall that, given the free abelian groups of countable rank
\begin{eqnarray*}
 E & = & \Z e_0 \oplus \Z e_1 \oplus \Z e_2 \oplus \cdots \\
 M & = & \Z me_0 \oplus \Z(e_1 - r_1( \xi)e_0) \oplus \Z(e_2 - r_2( \xi)e_0) \oplus \cdots \leqslant E \\
 M' & = & \Z me_0 \oplus \Z(e_1 - r_1( \xi')e_0) \oplus \Z(e_2 - r_2( \xi')e_0) \oplus \cdots \leqslant E \\
 E_1 & = & \Z e_1 \oplus \Z e_2 \oplus \cdots \leqslant E \ ,
\end{eqnarray*}
We have 
\begin{eqnarray*}
 \tBS m\xi & = & \Pres{a,E}{a\psi(x)a^{-1} = x \ \forall x\in E_1} \\
 \tBS m{\xi'} & = & \Pres{a,E}{a\psi'(x)a^{-1} = x \ \forall x\in E_1} \ , 
\end{eqnarray*}
where the isomorphism $\psi:E_1 \to M$ is defined by $\psi(e_1) = me_0$ and $\psi(e_{i+1}) = e_i -r_i( \xi)e_0$ for $i>0$, and the isomorphism $\psi':E_1 \to M'$ is defined similarly. Recall also that the element $b\in\bBS m\xi$ corresponds to $e_0\in\tBS m\xi$.
By Proposition \ref{PropRiSiCont}, the condition $\xi \equiv \xi' \ (\operatorname{mod}  \hat m^h d\Z_m)$ is equivalent to $r_i(\xi)=r_i(\xi')$ for $i=1,\dots, h$. We will consider the latter condition.

(1) Let $w=w(m,r_1(\xi),\dots,r_h(\xi))e_0w(-m,-r_1(\xi),\dots,-r_h(\xi))(-e_0)$ be defined as in Lemma \ref{LemWinE}. As $|w| \le 2(|m|+1)h + 2|m| + 6$ and $w=1$ in $\tBS{m}{\xi}$, we also have $w=1$ in $\tBS{m}{\xi'}$. We deduce from Lemma \ref{LemWinE} that $r_i(\xi)=r_i(\xi')$ for $i=1,\dots, h$.

(2) Let $w$ be a (freely reduced) word on the alphabet $\{a^{\pm 1},b^{\pm 1}\}$ satisfying $|w|\leqslant 2h$. By substituting occurences of $b^\alpha$ by $\alpha e_0$, we obtain a sequence $s=(x_0, a^{\varepsilon_1}, x_1,\ldots, a^{\varepsilon_k}, x_k)$, with $k\geq 0$, of length at most $2h$, where $\varepsilon_i=\pm 1$ and $x_i$ is an element of the subgroup $\Z e_0 \leqslant E$ for all $i$. What we have to show is that the product of the sequence $s$ vanishes in $\tBS m\xi$ if and only if it vanishes in $\tBS m{\xi'}$. 

We \emph{reduce} the sequence $s$ in the HNN-extension $\tBS m\xi$, that is we perform, as long as possible, substitutions of:
\begin{itemize}
 \item a subsequence $(a,x,a^{-1})$, with $x\in M$, by the element $\psi^{-1}(x) \in E_1$;
 \item a subsequence $(a^{-1},x,a)$, with $x\in E_1$, by the element $\psi(x)\in M$.
\end{itemize}
We then obtain a sequence $t=(y_0,a^{\delta_1}, y_1, \ldots, a^{\delta_l}, y_l)$, with $l\geqslant 0$, which is reduced in $\tBS m\xi$, and whose product in the latter group is equal to the product of $s$. The number of substitutions from $s$ to $t$ is trivially at most $h$. Therefore, it is easy to see that $t$ and the intermadiate sequences contain only $a^{\pm 1}$ letters and elements of the subgroup $\Z e_0 \oplus \cdots \oplus \Z e_h$. 

Now, we use the hypothesis $r_i( \xi) = r_i( \xi')$ for $i=1,\ldots, h$. Therefore, the relation
\[
 M \cap (\Z e_0 \oplus \cdots \oplus \Z e_h) = M' \cap (\Z e_0 \oplus \cdots \oplus \Z e_h) 
\]
holds and $\psi$ and $\psi'$ are equal in restriction to $\Z e_1 \oplus \cdots \oplus \Z e_{h+1}$. It is thus possible to reduce the sequence $s$ in $\tBS m{\xi'}$ by performing the \emph{same} substitutions as in $\tBS m{\xi}$. Hence, the sequences $s$ and $t$ have the same product in $\tBS m{\xi'}$. Moreover, the sequence $t$ is also reduced in $\tBS m{\xi'}$ --- if not, an argument similar to the above one would show that $t$ is not reduced in $\tBS m\xi$.

Finally, by structure theorems on HNN-extensions, the product of $t$ vanishes in $\tBS m{\xi}$ (resp. $\tBS m{\xi'}$) if and only if $l=0$ and $y_0 = 0$ in $E$. This concludes the proof of part (2).

\end{sproof}

We now turn to the case $d=1$, that is, to the case of invertible $m$-adic integers. Recall that the metric on $\G$ is given by $d(N_1,N_2) = e^{-\nu(N_1,N_2)}$ if $N_1 \neq N_2$, where $\nu(N_1,N_2) = \inf\{|w|: w\in N_1 \triangle N_2 \}$.

\begin{corollary}\label{DistDist}
 Let $h\in\N \setminus \{0\}$ and $\xi,\xi'\in\Z_m^\times$ such that $|\xi - \xi'|_m = |m|^{-h}$. Setting $x= \ovBS$ and $x'= \bBS m{\xi'}$, we have
 \[
   e^{-2(|m|+1)(h+1) - 2|m| - 6} \leqslant d(x,x')  \leqslant e^{-2h-1} \ .
 \]
\end{corollary}

\begin{proof}
 If $d(x,x') < e^{-2(|m|+1)(h+1) - 2|m| - 6}$, then $\ovBS$ and $\bBS m{\xi'}$ have the same relations up to length $2(|m|+1)(h+1) + 2|m| + 6$. Theorem \ref{CongDistLim}(1) gives then $|\xi - \xi'|_m \leqslant |m|^{-(h+1)}$.
 
 On the other hand, the relation $|\xi - \xi'|_m = |m|^{-h}$ implies $\xi\equiv \xi' \, (\operatorname{mod} m^h\Z_m)$. Theorem \ref{CongDistLim}(2) gives then $d(x,x') \leqslant e^{-2h - 1}$.
\end{proof}

\subsection{Hausdorff dimension estimates}
We set $f$ to be the inverse of the (bijective) map $\BBS_m:\Z_m^\times \to Z_m^\times$. We now show that $f$ and $f^{-1}=\BBS_m$ both satisfy a H\"older condition.
\begin{proposition}\label{Holder}
 For all $x,x'\in Z_m^\times$, we have
 \[
  |f(x) - f(x')|_m \leqslant C d(x,x')^\alpha
 \]
 where $\alpha = (2(|m|+1))^{-1}\log |m|$ and $C$ is some positive constant.
\end{proposition}

\begin{proof}
 Set $\xi = f(x)$ and $\xi'=f(x')$, so that $x=\bBS m\xi$ and $x'= \bBS m{\xi'}$, and write $|\xi-\xi'|_m = |m|^h$ with $h\in\N$. Let us treat the case $h\in\N \setminus \{0\}$ first. Using Corollary \ref{DistDist} (at the second line), we get: 
 \begin{eqnarray*}
  |f(x) - f(x')|_m & = & |\xi-\xi'|_m = e^{-h\log |m|} \\
  d(x,x') & \geqslant & e^{-2(|m|+1)h - 4|m| - 8} = C_1 e^{-2(|m|+1)h} = C_1 \left(e^{-h\log |m|}\right) ^{\frac{2(|m|+1)}{\log |m|}}
 \end{eqnarray*}
 with $C_1>0$. Consequently, we have $d(x,x') \geqslant
 C_ 1 |f(x) - f(x')|_m^{\alpha^{-1}}$, whence $|f(x) - f(x')|_m \leqslant C_2 d(x,x')^\alpha$ for some $C_2 >0$.
 
 Finally, in case $h=0$, that is $\xi \not\equiv \xi' \ (\operatorname{mod} m)$, there is a word
 \[
  a^2 b^m a^{-1} b^{-t} a^{-1} b \cdot a^2 b^{-m} a^{-1} b^{t} a^{-1} b^{-1} \ , \text{ with } 0\leqslant t \leqslant |m| - 1 \ ,
 \]
 which is trivial in one of the marked groups $x=\bBS m\xi,x'= \bBS m{\xi'}$ 
but not in the other one. This gives a constant $D>0$ such that $d(x,x') \geqslant D$, 
hence a constant $C_3>0$ such that $|f(x) - f(x')|_m = 1 \leqslant C_3 d(x,x')^\alpha$.
\end{proof}
\begin{proposition}\label{Holder2}
 For all $\xi,\xi'\in \Z_m^\times$, we have
 \[
  d(f^{-1}(\xi),f^{-1}(\xi')) \leqslant |\xi-\xi'|_m^\beta 
 \]
 where $\beta = 2(\log |m|)^{-1}$.
\end{proposition}
\begin{proof}
 Let us write $|\xi-\xi'|_m = |m|^{-h}$ with $h\in\N$. 
By corollary \ref{DistDist}, we have $d(f^{-1}(\xi),f^{-1}(\xi')) \leqslant e^{-2h-1}$ (note that for $h=0$ this is trivially true, since $\operatorname{diam}(\G) = e^{-1}$). Hence, we get
 \[
  d(f^{-1}(\xi),f^{-1}(\xi')) \leqslant e^{-2h} = (e^{-h\log|m|})^{2(\log|m|)^{-1}} = |\xi-\xi'|_m^\beta \ ,
 \]
 which concludes the proof.
\end{proof}

\begin{theorem}\label{ThDimZm}
 The Hausdorff dimension of $Z_m^\times$ satisfies: 
 \[
  \frac{\log |m|}{2(|m|+1)}  \leqslant \dim_H(Z_m^\times) \leqslant \frac{\log |m|}2 
 \]
 (for all $m$ such that $|m|\geqslant 2$).
\end{theorem}

\begin{proof}
It is well-known, and easy to show, that $\dim_H(\Z_m^\times)=1$ with respect to the metric chosen in Section \ref{SubSecZm}.
 Set $\alpha = (2(|m|+1))^{-1}\log |m|$ and $\beta = 2(\log |m|)^{-1}$, 
as in Propositions \ref{Holder} and \ref{Holder2}. Classical theory of Haussdorf dimension
 (see e.g. \cite[Pr. 2.3]{Fal03} or \cite[Th. 29]{Rog70})
and these propositions give 
$1 \leqslant \alpha^{-1} \dim_H(Z_m^\times)$ and $\dim_H(Z_m^\times) \leqslant \beta^{-1}$, hence the result.
\end{proof}

\begin{corollary}\label{CorDimNotZero}
 The Hausdorff dimension of $\G$ satisfies $\dim_H(\G) \geqslant \log(2)/6$. In particular, this dimension does not vanish.
\end{corollary}

We have estimated the Hausdorff dimension of the subspaces $Z_m^\times$, which are homeomorphic to the Cantor set (provided that $|m|\ge 2$). But many interesting subspaces of $\G$, or $\mathcal G_n$, appeared in the litterature, e.g:
\begin{itemize}
 \item the Cantor set of Grigorchuk groups \cite{Gri84}; many such groups have intermediate growth;
 \item the closure $\mathcal H_n\subseteq \mathcal G_n$ of non-elementary hyperbolic groups considered by Champetier \cite{Cham00};
 \item the minimal Cantor subset of $\mathcal G_3$ constructed by Nekrashevych \cite{Nek07}.
\end{itemize}
The first-named author has proved that the box-counting dimension (and hence the Hausdorff dimension) of the set of Grigorchuk groups vanishes \cite{Guy07}. It also holds for the set of Nekrashevych groups as these groups share similar contracting properties with the latter. In the case of hyperbolic groups, we do not know whether the Hausdorff dimension vanishes or not.

\section{Complexity of the word and conjugacy problems} \label{SecWP}
In this section, we study isomorphism invariants of groups originating from language theory, namely
the space complexity and the Turing degree of the word and conjugacy problems.
Our results are inspired by the works for Grigorchuk \cite{Gri84} and Garzon and Zalcstein \cite{GZ91}
on the word problem of Grigorchuk groups. First, we show that the space complexity of the word problem for $\ovBS$ is tightly related to the space complexity of the rational integer sequence $(r_i(\xi))$ (Proposition \ref{PropWPovBS}). Second, we show that the conjugacy problem for $\ovBS$ is Turing reducible to the word problem for $\ovBS$ (Corollary \ref{CorTuring}). For the sake of simplicity, our emphasis is on the space complexity of the word problem. Analogs of Proposition \ref{PropWPovBS} for time complexity and the conjugacy problem could be proved if one is prepared to more technicalities. 

\paragraph{Space complexity}
Let $\A$ be a set. We denote by $\Aa$ the set all strings (or words) on $\mathcal{A}$. Let $s \in \Aa$. We denote by $|s|_{\A}$ the string length of $s$, that is the number of symbols of $\A$ in $s$. We may simply write $|s|$ when the underlying set is clearly given by the context.
A set $L$ is a \emph{language} if it is a subset of $\Aa$ for some finite set $\A$ called \emph{alphabet}.

Let $G$ be a group and let $X$ be a finite generating set of $G$. We
denote by $WP(G,X)$ the set of strings $s \in {(X \cup X^{-
1})}^{\ast}$ such that $s=1$ in $G$, i.e. $s$ reduces to the trivial
element of $G$. The decision problem of membership in $WP(G,X)$ is
called the \emph{word problem with respect to $X$}. The Turing time
and space complexity of the language $WP(G,X)$ are group-theoretic
properties independant of $X$ \cite{MO85}; so $X$ will be omitted.

\begin{nb}
A Turing machine M is an \emph{off-line Turing machine} if it has a read-only input tape with endmarkers and finitely many semi-infinite storage tapes.
All Turing machines considered in this section are off-line Turing machines that halts on every input. We adress the reader to \cite{HU79} for the complete definitions of terms used in this section.
\end{nb}
Let M be an off-line Turing machine and let $f:\R_{+} \longrightarrow \N$ be a function. If for every input word of length $n$, the machine M scans at most $f(n)$ cells on any storage tape, then M is said to be an $f(n)$ \emph{space-bounded Turing machine}. We denote by $\DS(f)$ (resp. $\NS(f)$) the class of languages which are accepted by a deterministic (resp. non-deterministic) $f(n)$ space-bounded Turing machine. A language $L$ is \emph{recursive} if it is accepted by a Turing machine.
A function $g:\N^k \longrightarrow \N^l$ is a \emph{recursive function} if it can be computed by a Turing machine (the $k$ arguments $i_1,\dots ,i_k$ of $g$ are initially placed on the input tape separated by $1$'s, as $0^{i_1}10^{i_2}1 \cdots10^{i_k}$, the $l$ arguments are placed similarly in some output tape). A function $g:\N \longrightarrow \N^l$ is said to \emph{belong to} $\DS(f)$ (resp. $\NS(f)$) if there exists a deterministic (resp. non-deterministic) Turing machine taking as input the binary expansion of $j$ and computing $g(j)$ in space bounded above by $f(n)$ where $n$ is the number of binary digits of $j$.
A language $L$ (resp. a function $g$) is said to \emph{separate the inclusion} of two space complexity classes $$\DS(f) \subset \NS(f)$$
if $L$ (resp. $g$) belongs to $\NS(f)$ but not to $\DS(f)$.
Proofs below use of the Tape Compression Theorem \cite[Th. 12.1]{HU79} without mentioning it: the equality of language classes $$\DS(f)=\DS(cf)$$ holds for any $c>0$, with an analogue statement in the non-deterministic case.

Time complexity is analogously defined by counting the number of state transitions of a Turing machine with a read-and-write input tape. Every input word of length $n$ requires at least $n$ state transitions to be entirely read, hence $\DT(n)$ is the smallest time complexity class. For every function $f$, we have $\DT(f) \subset \DS(f)$.
We collect few facts on the word problem of finitely generated groups.
\begin{itemize}
\item The language $WP(G)$ is regular if and only if $G$ is a finite group \cite{Ani71}. If $G$ is infinite then $WP(G)$ does not belong to $\DS(\log \log)$ \cite{HS65}.
\item The language $WP(G)$ is context-free if and only if $G$ is virtually free \cite{MS83,Dun85}.
\item The language $WP(G)$ belongs to $\DS(\log)$ if $G$ is a linear group over a field of characteristic zero \cite{LZ77}. There exists a finitely presented non-linear group $G$ such that $WP(G) \in \DS(\log)$ \cite{Waa81}.
\item There is no known example of a ``simple'' group presentation for which the word problem does not belong to $\DS(\log)$.
\item If $G$ contains a copy of $\Z$ then $WP(G)$ does not belong to $\DS(g)$ for any $g$ such that $g(n)/\log(n)$ tends to $0$ \cite[Th. 2]{AGM92}. In particular, $\log$ is a sharp bound for the space complexity of the word problem of any infinite finitely generated linear group.
\item The word problem of a word hyperbolic group $G$ is solvable in real time \cite{Hol00}. In particular $WP(G) \in \DT(n)$.
\end{itemize}

Let $p,q \in \Z \setminus \{0\}$ and let $WP(p,q)$ (resp. in $WP(m,\xi)$) be the set of strings $s\in \{a^{\pm 1}, b^{\pm 1}\}^{\ast}$ such that $s=1$ in $BS(p,q)$ (resp. $\ovBS$).
Given $\ovBS$, we define the function $r$ on $\N$ by $r(0)=|m|$ and $r(n):=r_n(\epsilon_m  \xi)$ where $\epsilon_m$ is the sign of $m$. This definition is motivated by the fact that $WP(m,\xi)=WP(|m|,\epsilon_m \xi)$ since $\ovBS$ and $\overline{BS}(-m,-\xi)$ are isomorphic as marked groups. The following proposition can be proved by using arguments similar to those of Lemma \ref{LemSpaceBound}.

\begin{proposition} \label{PropWPBS}
$WP(p,q) \in \DS(n) \cap \DT(n^2)$.
\end{proposition}

As $BS(p,q)$ is not virtually free, we observe that the language $WP(p,q)$ is not a context-free language. The complement of $WP(p,q)$ is not a context-free language either, except if $|p|=|q|$ \cite{HRR+05}\footnote{It is uncorrectly claimed in the proof of \cite[Th. 13]{HRR+05} that $|p|=|q|$ if and only if $BS(p,q)$ is virtually abelian. The condition $|p|=|q|$ is less restrictive for it means that $BS(p,q)$ contains a copy of the direct product $\Free_{|p|} \times \Z$ as a finite index subgroup, or equally that $BS(p,q)$ is linear over $\Z$.}. Solvable Baumslag-Solitar groups (i.e., groups $BS(p,q)$ with $|p|=1$ or $|q|=1$) have a tidy real-time word problem \cite[Th. 2.1]{HR03}. We still ignore wether $WP(p,q)$ belongs to $\DS(\log)$ in the case $BS(p,q)$ is not linear. (Recall that $BS(p,q)$ is linear if and only if either $|p|=|q|$ or $|p|=1$ or $|q|=1$ by Proposition \ref{eqNBS}.) The reader interested in goedesic languages of Baumslag-Solitar groups should consult \cite{Eld05, DL09}.

Provided $r$ belongs to $\DS(n)$, Proposition \ref{PropWPBS} holds for $WP(m,\xi)$ and it corresponds to the lowest complexity bound we obtain. Our next result relate the space complexity of $WP(m,\xi)$ to the space complexity of $r$. Let us stress on the fact that functions $r$ are ``numerous'' because of Proposition \ref{PropRiSi}.$iii$: for any $m \in \Z\setminus\{0\},\,d \in \N \setminus\{0\}$ and for any $g:\N \longrightarrow \{0,\dots,|m|-1\}$ there is some $ \xi \in \Z_{m}^{\times}$ such that $g(n)=r_n(\epsilon_m  \xi)$ for all $n \ge 2$. Hence the following proposition can be seen as a result of density in the space hierarchy.
\begin{proposition} \label{PropWPovBS}
Let $f$ be a non-decreasing function such that $f(n) \ge n$ and $\DS(f) \neq \NS(f)$. Let $\ovBS$ be such that $r$ separates the inclusion $\DS(f) \subset \NS(f)$. Then $WP(m,\xi)$ separates the inclusion $$\DS(f(n/6|m|)) \subset \NS(f(n)).$$
\end{proposition}

This result is an immediate consequence of the following two lemmas.

\begin{lemma} \label{LemSpaceBound}
Assume $r \in \DS(f(n))$ for some non-decreasing function $f$. Then $WP(m,\xi) \in \DS(n+f(n))$. Likewise for $\NS$.
\end{lemma}
\begin{lemma} \label{LemSpaceBoundInf}
Assume $WP(m,\xi) \in \DS(f(n/6|m|))$ for some non-decreasing function $f$ such that $f(n) \ge n$. Then $r \in \DS(f(n))$. Likewise for $\NS$.
\end{lemma}
Let us summarize the idea of the proof of Lemma \ref{LemSpaceBound}.
Applying to a given word $w\in \ab$ the natural algorithm originating from Britton's lemma, we obtain a reduced sequence for $w$. This reduction is carried out within at most $|w|_a$ steps and at each step we consider a word whose length is at most $|m|$ times the length of the previous one. As we encode the exponents of $a$ and $b$ by means of their binary expansions, this streching factor becomes an additive constant which explains the linear part of the space complexity bound. The other part of the bound comes from the fact that we need to compute $r(n)$ to reduce words $w$ such that $|w|_a=n$.

As for the proof of Lemma \ref{LemSpaceBoundInf}, we notice that a Turing machine which can solve the word problem for $\ovBS$, can decide which of the words defined in Lemma \ref{LemWinE} are trivial. Hence it can be used to compute $r_i(\xi)$ for every $i$.

We will work with our favoured HNN extension $\tBS{m}{\xi}$ instead of $\ovBS$. In order to make a careful enough counting of the numbers of scanned cells, we will use following notations.
We fix $\mathcal{A}:=\{a^{\pm 1},\pm e_0,\pm e_1,\dots\}$. Let $m \in \Z \setminus \{0\}$ and $\xi \in \Z_m$.
Given $w \in\mathcal{A}^{\ast}$, we can rewrite $w$ in $\langle  a \rangle*E$ under the form
\begin{equation} \label{EqForm}
 w^{(0)}= a^{\alpha_1} c_1 a^{\alpha_2}c_2 \cdots a^{\alpha_h} c_h
\end{equation}
with $c_j=(\beta_{0j}e_0) (\beta_{1j}e_1) \cdots (\beta_{k_j j}e_{k_j}),\,\alpha_j,\beta_{lj} \in \Z$ for all $l,j$.
We denote by $\eps_j$ the sign of $\alpha_j$. We suppose that the following holds: there is some $j$ such that
\begin{equation} \label{C} \tag{*}
\eps_j=-\eps_{j+1}=-1 \text{ and } c_j \in E_{m,\xi} \text{ or } \eps_j=-\eps_{j+1}=1 \text{ and } c_j \in E_1.
\end{equation}

We denote by $\ell=\ell(w)$ the smallest $j$ such that
(\ref{C}) holds. Let $w'$ be the word we get from $w^{(0)}$
by replacing $a^{\alpha_{\ell}} c_{\ell} a^{\alpha_{\ell+1}}$ by $a^{\alpha_{\ell}-\eps_{\ell}} \phi^{\eps_{\ell}} (c_{\ell})
 a^{\alpha_{\ell+1}+\eps_{\ell}}$ in $w$ and reducing this new word as 
in \ref{EqForm}. We write $w'=a^{\alpha_1'}c_1'a^{\alpha'_2} \cdots a^{\alpha'_{h'}}c'_{h'}$. Notice that a given exponent $\alpha$ 
of $a$ in $w$ either remains unchanged in $w'$, vanishes or is replaced by some $\alpha'$ such that $|\alpha'|=|\alpha|-1$. The subwords $c_j$ remain unchanged in $w'$ or vanish, except one which is replaced by some subword $c'$ with $|c'|\le(2+ \vert m\vert)n$ where $n=|w|$.
As long as (\ref{C}) holds for $w^{(i)}=a^{\alpha_1^{(i)}}c_1^{(i)}a^{\alpha_2^{(i)}} \cdots a^{\alpha^{(i)}_{h^{(i)}}}c_{h^{(i)}}^{(i)}$ with $i \ge 0$, we can define $w^{(i+1)}=(w^{(i)})'$.

By Britton's Lemma, for any $w \in \Aa$, there is some $i=i(w) \ge 0$ such that $w^{(i)}=1$ is a reduced form for $w$. We call the previous algorithm the Britton's algorithm.

\begin{sproof}{ Lemma \ref{LemSpaceBound}}

By hypothesis, there is an $f(n)$ space-bounded Turing machine
$\text{M}_r$ computing $r(n)$. We denote by $R_1$
its input tape and by $R_k$ $(2 \le k \le p)$ its storage tapes.
We design an off-line Turing machine M that halts on every input $w \in \aeo$: if a
non-trivial reduced form for $w$ has been found, it halts without
accepting, else $w$ is reduced to $1$ and M halts in an accepting state. Tape I is the read-only input tape
where $w$ is displayed without accounting for any space.
At the beginning, M writes the string $s^{(0)}$ on Tape 0 that encodes $w^{(0)}$:
$$s^{(0)}:=\mbox{\textcent} \eps_1\overline{\alpha}_1 \overline{c}_1 \eps_2\overline{\alpha}_2 \overline{c}_2
 \dots \eps_h\overline{\alpha}_h \overline{c}_h \$.$$
The strings $\overline{\alpha}_j\in \{0,1\}^{\ast}$ are the binary expansions of $|\alpha_j|$; each string
 $\overline{c}_j \in\{0,1,\pm \}^{\ast}$ is the concatenation of the binary expansions of the numbers
 $|\beta_{lj}|$ separated by sign symbols. If $\alpha_1=0$ (respectively $c_h=0$) then $\eps_1\overline{\alpha}_1$
  (respectively $\overline{c}_h$) is replaced by the empty string.

We now describe how M works on its storage tapes $R_k, \,(1 \le k \le p)$, $T_0,\,T_1$ and D.
First, the machine read the input: while
the head of tape I scans the first $j$ symbols of $w$, M stores the number $j$ using a counter situated in tape $R_1$ and then M computes and stores $r_j(\epsilon_m  \xi)$
in some of the tapes $R_k$ by simulating $M_r$. Meanwhile, the head of tape $T_0$ writes $s^{(0)}$,
 following an obvious linearly space-bounded algorithm. Once the input is read, M goes ahead by running
 Britton's algorithm. During the $i$-th step of this algorithm, with $i$ even, the head of $T_0$
 writes the string $s^{(i)}$ encoding $w^{(i)}$ over $s^{(i-2)}$ if condition (\ref{C}) holds for $w^{(i-1)}$.
  The latter word is encoded by a string $s^{(i-1)}$ stored in Tape $T_1$. In the next step,
   the head of Tape $T_1$ writes the string $s^{(i+1)}$ encoding $w^{(i+1)}$ over $s^{(i-1)}$
   if condition (\ref{C}) holds for $w^{(i-1)}$.
Tape D is a draft tape used to carry
 out two kind of arithmetical computations on binary expansions: the tests for condition
 (\ref{C}) and the computations of $c^{(i)}$. The content of Tape D is erased after each step.
The machine M halts in a state of acceptation if $s^{(i)}$ is the trivial string.
It halts without accepting in case condition (\ref{C}) does not hold for $w^{(i)}$.

{\it Space bound.}
The machine M scans at most $f(n)$ cells on the storage tapes $R_k$ while computing $r_1(\epsilon_m  \xi),\dots,r_n(\epsilon_m  \xi)$.
It also scans at most $C_0n$ cells while storing each number $j$ and all numbers $r_j(\epsilon_m  \xi)$ for $j \le n$, where $C_0>0$ is
independent of $n$.

Since $|s^{(i+1)}| \le |s^{(i)}|+\log_2(|m|)$ and $|s^{(0)}| \le 2n+2$, we deduce that M scans at most $C_1n$ cells on the storage
 tapes $T_0$ and $T_1$, where $C_1>0$ is independent of $n$.
In order to decide if $c_j^{(i)}$ belongs to $E_{m,\xi}$ or $E_1$, M uses the formula of Proposition \ref{PropBandBm}.$iii$: according to the signs of $\alpha_j^{(i)}$ and $\alpha_{j+1}^{(i)}$, M carries out the division of $\gamma_{m}:=\beta^{(i)}_{0j}+\beta_{1j}^{(i)}r_1(\epsilon_m  \xi)+\beta_{2j}^{(i)} r_2(\epsilon_m  \xi)+\dots+\beta_{k_j j}^{(i)}r_{k_j}(\epsilon_m  \xi))$ by $|m|$ or divides $\gamma_{\xi}:=\beta^{(i)}_{1j}+\beta_{2j}^{(i)}r_1(\epsilon_m  \xi)+\beta_{2j}^{(i)} r_2(\epsilon_m  \xi)+\dots+\beta_{k_j j}^{(i)}r_{k_j}(\epsilon_m  \xi)$ by $|m|$ if moreover $\beta_{0j}^{(i)}=0$.
As $\log_2(1+|\gamma|) \le |s^{(i+1)}|$, for $\gamma=\gamma_m,\gamma_{\xi}$, this requires to scan at most $C_2n$ cells on Tape D, where $C_2>0$ is independent of $n$. In order to compute $c^{(i)}$, no more than $|s^{(i+1)}|$ cells need to be scanned on Tape D. Hence the number of cells scanned by M on Tape D is linearly bounded. All in all, we get $WP(m,\xi) \in \DS(n+f(n))$.
\end{sproof}

The first part of the proof of Lemma \ref{LemSpaceBoundInf} is based on the following facts.

\begin{lemma} \label{LemGetM}
Let $m,n\in\Z\setminus\{0\}$ with $\vert m \vert>1$ and let $\xi \in
\Z_m$. Let $v_k=\lbrack ab^ka^{-1},b \rbrack$ for $k\in \Z$. Then we
have: $v_k=1$ in $\ovBS$ if and only if $k \equiv 0 \, (\operatorname{mod} m\Z
)$.

\end{lemma}

\begin{proof} \noindent
 Let $\vert n \vert >1$. We deduce from Britton's lemma the following claim: for
every $k \in \Z$, we have $v_k=1$ in $BS(m,n)$ if and only if $k
\equiv 0 \, (\operatorname{mod} m\Z )$. As $BS(m,\xi_n)$ tends to $\ovBS$ as $n$ goes to
infinity, $v_k$ is trivial in $\ovBS$ if and only if it is
trivial in $BS(m,\xi_n)$ for all $n$ large enough, which completes the proof.

\end{proof}
For $h \ge 1,t_1,\dots,t_h \in \{0,\dots,|m|-1\}$, we set $$v(|m|,t_1,\dots,t_h):=w(|m|,t_1,\dots,t_h)bw(-|m|,-t_1,\dots,-t_h)b^{-1}$$
where $w(|m|,t_1,\dots,t_h)$ is defined as in Lemma \ref{LemWinE}. 

\begin{sproof}{Lemma \ref{LemSpaceBoundInf}}
By hypothesis, there is a deterministic $f(n/6|m|)$ space-bounded Turing machine M that solves the word problem for $\ovBS$. We design a Turing machine $\Mp$ computing $r(n)$ as follows. The storage tapes of $\Mp$ consists of the tapes of M and two other tapes W, and O (output tape). The tape W identifies with the input tape of M and $\Mp$ simulates M on every tape of M.

{\it Computation of $|m|$.}
By Lemma \ref{LemGetM},
we have $|m|=\min \{k\ge 1:\, v_k=1 \text{ in } \ovBS \}$ 
The machine $\Mp$ first writes $v_k$ on tape W for $k=1$ and runs M. While $v_k$ is not accepted by M, the machine $\Mp$ writes $v_{k+1}$ over $v_k$, adds one to a counter storing $k$ in tape O and clears the storage tapes of M. If $v_k$ is accepted, which means $k=|m|$, then $\Mp$ clears tape W.

{\it Computation of $r_n(\epsilon_m  \xi)$.}
Using two counters that store $i \le n$ and $t \in \{0,\dots,|m|-1\}$ in tape O, the machine $\Mp$ lists recursively the words $w_i(t)=v(|m|,r_1(\epsilon_m\xi),\dots,r_{i-1}(\epsilon_{m}\xi),t)$ on tape W. Once a word $w_i(t)$ is written on tape W, the machine $\Mp$ runs M. If $w_i(t)$ is not accepted by M, then $\Mp$ writes $w_i(t+1)$ over $w_i(t)$, clears the storage tapes of M and runs M again. If the word written on $W$ is accepted by M, which means $t_i=r_i(\epsilon_m  \xi)$ by Lemma \ref{LemWinE}, then $\Mp$ stores $r_i(\xi)$ in tape O, increment $i$ and restarts with $w_{i+1}(0)$ or halts if $i=n$.

{\it Space bound.}
Obviously, the number of cells scanned by $\Mp$ to compute $|m|$ is bounded by some constant $C(m)$ independent of $n$.
The number of cells scanned by $\Mp$ while writing words $v(|m|,t_1,\dots,t_n)$ on tape W is bounded by $|v(|m|,t_1,\dots,t_n)| \le 6n|m|$, the number of cells used to store $r_1(\xi),\dots,r_n(\xi)$ is bounded by $n \log_2(|m|+1)$ and the number of cells scanned by $\Mp$ while simulating M over its storage tapes is bounded by $f(n)$. Hence $WP(m,\xi) \in \DS(f)$.
\end{sproof}

\paragraph{Turing degree}

Let $E,F$ be languages. The language $E$ is said to be \emph{Turing reducible to} $F$ if there is Turing machine M with oracle $F$ whose accepted language is $E$. The language $E$ is said \emph{Turing equivalent to} $F$ if $E$ is Turing reducible to $F$ and $F$ is Turing reducible to $E$. The \emph{Turing degree of} $E$ (also called the \emph{degree of unsolvability} of $E$) is the class all languages that are Turing equivalent to $E$. Let $f:E \longrightarrow F$ be a function. We define the \emph{Turing degree of $f$} as the Turing degree of the graph of $f$. We denote by $CP(m,\xi)$ the set of pairs  $(w,w') \in \{a^{\pm 1},b^{\pm 1}\}^{\ast} \times \{a^{\pm 1},b^{\pm 1}\}^{\ast}$ such that $w$ is conjugated to $w'$ in $\ovBS$. We call the Turing degrees of $WP(m,\xi),CP(m,\xi)$ the Turing degrees of the word problem and the conjugacy problem for $\ovBS$. These Turing degrees does not depend on the choice of a generating set for $\ovBS$.
\begin{corollary} \label{CorTuring}
The following Turing degrees coincide:
\begin{itemize}
\item the Turing degree of the word problem for $\ovBS$;
\item the Turing degree of the conjugacy problem for $\ovBS$;
\item the Turing degree of $r$.
\end{itemize}
In particular, the word problem is solvable for $\ovBS$, i.e $WP(m,\xi)$ is a recursive language, if and only if $r$ is a recursive function.
\end{corollary}

In contrast, Britton has proved that the conjugacy problem for any HNN extension with base a finitely generated abelian group is solvable, i.e. both Turing degrees are $\mathbf{0}$ \cite{Bri79}. It is also worth noting this optimal result of Miller: for every pair of recursively enumerable Turing degrees $\mathbf{a},\,\mathbf{b}$ where $\mathbf{a}$ is Turing reducible to $\bf{b}$, there is a finitely presented group whose word problem has Turing degree $\mathbf{a}$ and whose conjugacy problem has Turing degree $\mathbf{b}$ \cite{Mil71}.

Observe that one can define recursive $m$-adic numbers in the very same way one defines recursive (equivalently computable) real numbers (see \cite{Wei00} for a definition of computable real numbers). The Turing degree of an $m$-adic number is then defined by means of its Hensel expansion. If $\xi \in \Z_m^{\times}$, Corollary \ref{CorTuring} then says that the word problem is solvable in $\ovBS$ if and only if $\xi$ is a recursive number and that the Turing degree of the word problem coincides with the Turing degree of $\xi$.

\begin{proof}
The last claim directly follows from Lemmas \ref{LemSpaceBound} and \ref{LemSpaceBoundInf}. From the proofs of these lemmas, we can easily deduce that the Turing degree of $WP(m,\xi)$ coincides with the Turing degree of $r$. 

To complete the proof we design quite informally a Turing machine with oracle $r$ that solves the conjugacy problem in $\ovBS$.
We fix the set of reprensatives $T_{m,\xi}=\{0,e_0,\dots,(m-1)e_0\}$ of the cosets of $E_1$ in $E$ and the set of representatives  $T_{1}=\{0,e_0,2e_0,\dots\}$ of the cosets of $E_{m,\xi}$ in $E$. If $r$ can be computed by means of a Turing machine, Britton's algorithm (see the proof of Lemma \ref{LemSpaceBound}) yields a reduced form in $\tBS{m}{\xi}$ of any $w \in \aeo$. The process of working from the right with the relations $a(me_0)a^{-1}=e_1$ and $a(e_i-r_i(\xi)e_0)a^{-1}=e_{i+1}$ yields a normal form for $w$ with respect to the sets of representatives $T_{m,\xi}$ and $T_{1}$. Thus we can design a Turing machine with oracle $r$ that computes normal forms $\tv,\tw$ of cyclically reduced conjugates of $v$ and $w$ for any $v,w \in \aeo$. If $|\tv|_a \neq |\tw|_a$, we deduce from Collin's lemma \cite[Th. 2.5]{LS77} that $v$ is not a conjugate of $w$ in $\tBS{m}{\xi}$. The machine can be designed in such a way that it halts in this case in a non-accepting state. Hence we can assume that $|\tv|_a=|\tw|_a$.
Comparing the normal form $\tv$ to the normal form of each cyclic permutation of $\tw$, the machine can decide wether or not there exist, $e \in E$ and some cyclic permutation $(\tw)^{\ast}$ of $\tw$ such that $\tv=e(\tw)^{\ast} (-e)$. By Collin's lemma, it is enough to decide wether $v$ is a conjugate of $w$, provided either $|\tv|_a$ or $|\tw|_a$ is not zero.
Hence we can assume that $v,w$ have their images in $E$. We deduce from Lemma \ref{LemFixEll} that $v$ is a conjugate of $w$ in $\tBS{m}{\xi}$ if and only if there is some $n \in \Z$ such that $v=a^nwa^{-n}$ in $\tBS{m}{\xi}$. Identifying $E$ with $B$ in Proposition \ref{PropBandBm}.$i$, we can consider $n(v,w)=\deg P_{v}(X)-\deg P_w(X)$. By means of a Turing machine with oracle $r$, we can compare the Laurent polynomials $P_{v}(X)$ and $X^{n(v,w)}P_{w}(X)$ and hence decide wether or not $v$ is a conjugate of $w$.

\end{proof}

\begin{rem}
We can construct a family of public-key cryptosystems based on the word problems of limits of Baumslag-Solitar groups by adaptating the construction  in \cite{GZ91} based on Grigorchuk groups. The attack conceived in \cite{GHM+04} does not threaten these new cryptosystems since such an attack would require in our case at least $|m|^N$ numbers of computations when the length of the public-key is $N$, if we follow the cryptanalysis of the authors. However, another attack conceived in \cite{GHM+04}, namely the reaction attack against the Magyarik-Wagner protocol, can be proved to be successful. 
\end{rem}

\bibliographystyle{alpha}

\medskip

Authors addresses:

\medskip

L. G. Mathematisches Institut, Georg-August Universit\"at,
Bunsenstrasse 3-5, G\"ottingen 37073, Germany,
guyot@uni-math.gwdg.de
\medskip

Y. S. Clermont Universit\'e, Universit\'e Blaise Pascal, Laboratoire de Math\'ematiques, BP 10448, F-63000 Clermont-Ferrand, France ---
CNRS, UMR 6620, Laboratoire de Math\'ematiques, F-63177 Aubi\`ere, France, yves.stalder@math.univ-bpclermont.fr

\end{document}